\documentclass[11pt]{amsart}
\usepackage[english]{babel}
\usepackage{latexsym}
\usepackage{amssymb}
\usepackage{amsmath}
\usepackage{amsfonts}
\usepackage[mathscr]{eucal}
\usepackage{MnSymbol}
\usepackage{color}
\usepackage[dvipsnames]{xcolor}
\usepackage{amscd}
\usepackage[all,cmtip]{xy}
\usepackage{graphicx}
\usepackage{epsfig}
\usepackage{amsthm}
\usepackage[pdfstartview=FitH]{hyperref}
\usepackage{pst-grad}
\usepackage{pst-plot} 
\usepackage{comment}
\usepackage{psfrag}
\usepackage{verbatim}
\usepackage{soul}
\usepackage{lineno} 

\newtheorem{thm}{Theorem}[section]
\newtheorem{cor}[thm]{Corollary}
\newtheorem{lem}[thm]{Lemma}
\newtheorem{prop}[thm]{Proposition}
\newtheorem{rmrk}[thm]{Remark}
\newtheorem{theorem}{Theorem}

\theoremstyle{definition}
\newtheorem{definition}[thm]{Definition}

\newtheorem{example}[thm]{Example}

\numberwithin{equation}{section}
\DeclareMathOperator{\dist}{dist}
\newcommand{\vol}{\operatorname{Vol}}

\newcommand{\R}{\mathbb{R}}
\newcommand{\vare}{\varepsilon}
\newcommand{\diam}{\operatorname{diam}}
\newcommand{\Slice}{\operatorname{Slice}}
\newcommand{\hm}{{\mathcal H}}
\newcommand{\set}{\operatorname{set}}
\newcommand{\intcurr}{{\mathbf I}}      
\newcommand{\SF}{{\mathbf {SF}}}
\newcommand{\fillvol}{{\operatorname{FillVol}}}
\newcommand{\rstr}{\:\mbox{\rule{0.1ex}{1.2ex}\rule{1.1ex}{0.1ex}}\:}
\newcommand{\bdry}{\partial}
\newcommand{\mass}{{\mathbf M}}

\begin{document}
\title[A generalized tetrahedral property]{A generalized tetrahedral property}
\author[J.~N\'u\~nez-Zimbr\'on]{Jes\'us N\'u\~nez-Zimbr\'on$^{\ast}$} 
\author[R.~Perales]{Raquel Perales $^{\ast \ast}$ }
\address[J.~N\'u\~nez-Zimbr\'on]{Department of Mathematics, University of California, Santa Barbara, California, 93106, USA}
\curraddr{Centro de Ciencias Matem\'aticas, UNAM, Campus Morelia, 59089, Morelia, Michoac\'an, M\'exico}
\email{zimbron@matmor.unam.mx}
\address[R.~Perales]{Conacyt Research Fellow. Instituto de Matem\'aticas, Universidad Nacional Aut\'onoma de M\'exico\\ Oaxaca, M\'exico.}
\email{raquel.peralesaguilar@gmail.com}
%
\begin{abstract}
We present examples of metric spaces that are not Riemannian manifolds nor dimensionally homogeneous that satisfy Sormani's Tetrahedral Property. We then note that Euclidean cones over metric spaces with small diameter do not satisfy this property. Therefore, we extend the tetrahedral property to a less restrictive one and prove that this generalized definition retains all the results of the original tetrahedral property proven by Portegies-Sormani: it provides a lower bound on the sliced filling volume and a lower bound on the volumes of balls.  Thus, sequences with uniform bounds on this Generalized Tetrahedral Property also have subsequences which converge in both the Gromov-Hausdorff and Sormani-Wenger intrinsic flat sense to the same noncollapsed and countably rectifiable limit space.
\end{abstract}
\maketitle

\section{Introduction}

\subsection{Historical context}

The $n$-dimensional $(C, \beta)$-tetrahedral property for a metric space was originally defined in \cite{Sor1}. As this property is rather strong an integral version was also introduced. These properties are given in terms of distances between points in metric spheres and provide a lower bound on the volumes of balls. Thus, a Gromov-Hausdorff (Integral) Tetrahedral Compactness Theorem for sequences of $n$-dimensional Riemannian manifolds with a uniform upper bound on volume and diameter that satisfy a uniform (integral) tetrahedral property was deduced. Furthermore, it is shown that the limits are countably $ \mathcal{H}^n$-rectifiable metric spaces \cite{PorSor}, where $\mathcal{H}^n$ denotes the $n$-dimensional Hausdorff measure. 

The proof of the (Integral) Tetrahedral Compactness Theorems is based upon intrinsic flat convergence.
The intrinsic flat distance between integral current spaces, countably $\mathcal{H}^n$-rectifiable metric spaces that are generalizations of oriented manifolds, was introduced by Sormani-Wenger in \cite{SW}. It was defined using Gromov's idea of isometrically embedding two metric spaces into a common metric space. However, rather than measuring the Hausdorff distance between the images, since they are integral currents in the sense of Ambrosio-Kirchheim \cite{AK} one takes the flat distance between them. A compactness theorem with respect to the intrinsic flat distance for the class of $n$-dimensional integral current spaces holds \cite{Wen}. Thus, limits obtained with this distance are automatically either countably $\mathcal{H}^n$-rectifiable or the $n$-dimensional zero integral current space. Moreover, once a sequence of integral current spaces converges in the Gromov-Hausdorff sense to a limit space then a subsequence converges in the intrinsic flat sense to a subset of the aforementioned limit \cite{SW}.  Thus, one can show that Gromov-Hausdorff limits are countably rectifiable by showing both limits agree.  Now, to proceed in this way since the mass of the currents is lower semi continuous with respect to the intrinsic flat distance one has to bound quantities that behave better with respect to this distance.

Sormani-Wenger proved, using a filling volume estimate for spheres, that the two limit spaces agree for noncollapsed sequences of closed, oriented $n$-dimensional manifolds that have nonnegative Ricci curvature and a uniform upper volume bound \cite{SW}.  The second named author proved a statement similar to that of Sormani-Wenger for manifolds with boundary in \cite{Per} and Li and the second named author proved an analogous result for Alexandrov spaces in \cite{LiPer}.  
Matveev-Portegies, using a different approach, extended Sormani-Wenger result allowing any uniform lower bound on the Ricci curvature \cite{MP}.  Portegies-Sormani's  (Integral) Tetrahedral Compactness Theorem shows that both limits agree by introducing a new notion called the sliced filling volume of a ball and the $k$-th sliced filling of balls. These quantities behave well with respect to intrinsic flat distance. Instead of showing strong estimates on the filling volumes of spheres as Sormani-Wenger did, Portegies-Sormani  prove that, at points where the (integral) tetrahedral property holds, the mass measure of balls is bounded below by the $k$-th sliced filling volume. The latter is bounded below by an appropriate power of the radius of the ball times a constant coming from the constants appearing in the tetrahedral property. From this, a Gromov-Hausdorff precompactness result is derived, further obtaining that the intrinsic flat limit is not only contained in the Gromov-Hausdorff limit but in fact both limits are equal \cite{PorSor}. 

We recall that by the work of Cheeger-Colding it was already known that noncollapsed Gromov-Hausdorff limits of $n$-dimensional manifolds that have Ricci curvature bounded below were countably $\mathcal{H}^n$-rectifiable \cite{CheCol}.

In this paper we extend Sormani's tetrahedral property to a less restrictive one and prove that this definition retains all the results of the original property.  As a matter of fact, we also define an integral version as was done by Sormani. Nonetheless, in this case, both integral versions are equivalent up to a transformation of the various constants.

\subsection{Results}

We start by recalling the  \textit{tetrahedral property} which was previously defined in \cite{PorSor} and \cite{Sor1}. Let $(X,d)$ be a metric space. Here, $B_r(p)$ denotes the open metric ball of radius $r$ centered at $p$ and $\bar X$ the metric completion of $X$.  Let $S(p;r)=\{ x \in X | \,d(x,p)=r \}$ and $S(x_1,\dots,x_j;t_1,\dots,t_j) = \bigcap_{i=1}^j   S(x_i;t_i)$. If $A$ is a set we denote its cardinality by $|A|$. 

\begin{definition}\label{def-Tprop}[Sormani]
Let $C>0$ and $\beta\in (0,1)$. A metric space $(X,d)$ has the \textit{$n$-dimensional $(C,\beta)$-tetrahedral property at a point $p$ for radius $r$} if there exist points $p_1,\ldots, p_{n-1}\in \bar{X}$, such that $d(p,p_i)=r$ and for all  $(t_1,\ldots, t_{n-1})\in [(1-\beta)r,(1+\beta)r]^{n-1}$ 
\begin{equation}
h(p,r,t_1,\ldots, t_{n-1})\geq Cr,
\end{equation}
where 
\begin{equation}
h(p,r,t_1,\ldots, t_{n-1})=\left\{
\begin{array}{ll}
      \inf\{ \, d(x,y) \mid x\neq y, x,y\in S \}  & |S| \geq 2 \\
      0 & \text{otherwise} \\
\end{array} 
\right.
\end{equation}
and
 \begin{equation} 
S= S(p,p_1,\dots,p_{n-1};r,t_1,\dots,t_{n-1}).
\end{equation}
We say that $X$ \textit{satisfies the $n$-dimensional $(C,\beta)$-tetrahedral property for radius r} if it satisfies the $n$-dimensional $(C,\beta)$-tetrahedral property at every point for radius $r$. 
\end{definition}
In general, $C,\beta$ and $r$ depend on $p \in X$. Whenever this might cause confusion we will make explicit mention of the dependence. 

The only examples in the literature of metric spaces that satisfy the tetrahedral property are Example 2.1, Example 2.2 and Example 2.3 in \cite{Sor1}. They are:  the three-dimensional Euclidean space, tori of the form $S^1 \times S^1 \times S^1_\varepsilon$ where the radius for which the property is satisfied goes to zero as $\varepsilon$ goes to zero, and two copies of Euclidean space with a large collection of tiny necks between corresponding points. The first two also appear as Example 3.31 and Example 3.32 in  \cite{PorSor}.  In this work, we present Example \ref{ex-planes} and Example \ref{ex-plane&Line} that show that the metric spaces satisfying the tetrahedral property do not need to be Riemannian manifolds. In fact, such spaces do not need to be dimensionally homogeneous.

In Definition \ref{def-Tprop}, each $t_i$ takes values on the interval $[(1-\beta)r,(1+\beta)r]$. In particular, $t_i=r \in [(1-\beta)r,(1+\beta)r]$. This implies that  Euclidean cones over metric spaces $(X,d)$ with $\diam(X) < \pi/3$ do not satisfy the $(C, \beta)$-tetrahedral property,  Example \ref{ex-conesModT}. This example implies that the class of $n$-dimensional Alexandrov spaces with nonnegative curvature do not satisfy a uniform tetrahedral property, Corollary \ref{cor-AlexNoUnif}.  The same is true for the class of $n$-dimensional Riemannian manifolds with no negative sectional curvature.
This answers in the negative the conjecture in Remark 3.35 of \cite{PorSor}.

By allowing $t_i$ to take values on intervals that do not necessarily contain $r$, manifolds with conical singularities and cones which failed to satisfy the tetrahedral property at their tips for radius $r$ satisfy the $(C,\alpha,\beta)$-tetrahedral property, Example \ref{ex-conesModT2}.  
Nonetheless, it is still open whether Alexandrov spaces satisfy the tetrahedral property or the generalized one, Remarks  \ref{rmrk-AlexSpTprop}-\ref{rmrk-AlexSpTprop2}.

\begin{definition}[$(C,\alpha,\beta)$-tetrahedral property]\label{def-modTprop}
Let $C>0$ and $\alpha, \beta\in (0,2)$, $\alpha < \beta$.  A metric space $(X,d)$ satisfies the $n$-dimensional $(C,\alpha,\beta)$-tetrahedral property at a point $p$ for radius $r$ if there exist points $p_1,\ldots, p_{n-1}\in \bar{X}$ such that $d(p,p_i)=r$ and for all $(t_1,\ldots, t_{n-1})\in [\alpha r,\beta r]^{n-1}$ the following holds  
\[ 
h(p,r,t_1,\ldots, t_{n-1})\geq Cr,
\]
where $h$ is given as in Definition \ref{def-Tprop}.

We say that $X$ \textit{satisfies the $n$-dimensional $(C,\alpha,\beta)$-tetrahedral property for radius r} if it satisfies the  $n$-dimensional $(C,\alpha, \beta)$-tetrahedral property at every point for radius $r$. 
\end{definition}

For any $\beta \in (0,1)$, the $(C,\beta)$-tetrahedral property implies the $(C,1-\beta,1+\beta)$-tetrahedral property. The $(C,\alpha,\beta)$-tetrahedral property implies a $(C,\beta')$ property only if $\alpha < 1 < \beta$.   In this paper we show that this new definition retains all the powerful properties of the original tetrahedral property that were proven by Portegies-Sormani in \cite{PorSor}, including a lower bound for the sliced filling volume. Thus, sequences of manifolds with uniform lower bounds on our revised tetrahedral property have subsequences which converge in the Gromov-Hausdorff and intrinsic flat sense to noncollapsed rectifiable limit spaces. See Theorem 3.41  and Theorem 5.2 in \cite{PorSor}, cf. Theorem \ref{thm-gh} and Theorem \ref{compGH=IF}. 

For ease of notation we present in this section results for manifolds. Inside we prove versions for integral current spaces. The next theorem follows from Theorem  \ref{tetra-ball} where we estimate the mass and the $(n-1)$-th sliced filling obtaining,
\begin{equation*}
\mass(S(p,r))\geq \SF_{n-1}(p,r) \geq C (\beta- \alpha)^{n-1}r^n.
\end{equation*}

\begin{theorem}\label{tetra-manifold}
Let $M$ be a compact oriented Riemannian manifold possibly with boundary. Suppose that $p \in M$ and $B_{R}(p)\cap \partial M=\emptyset$. If for almost every $r\in (0,R)$ the $n$-dimensional $(C,\alpha,\beta)$-tetrahedral property at $p$ for radius $r$ holds, then
\begin{equation}
\vol(B_r(p))\geq C (\beta - \alpha)^{n-1}r^n.
\end{equation}
\end{theorem}

The previous volume estimate implies the following Gromov-Hausdorff convergence result. 

\begin{theorem}\label{tetra-compactness2} 
Let  $r_0>0$, $0< \alpha <  \beta < 2$,  $C>0, V_0>0$ and $M_i$ a sequence of $n$-dimensional closed oriented and connected Riemannian manifolds
that satisfy 
\[
\vol(M_i) \le V_0. 
\]
If all $M_i$ satisfy the $n$-dimensional $(C,\alpha,\beta)$-tetrahedral property
for all radii $r \leq  r_0$. Then a subsequence of the
$M_i$ converges in Gromov-Hausdorff sense to a noncollapsed metric space. In particular, there exists $D_0(C,\alpha,\beta, r_0, V_0) > 0$ such that 
$ \diam(M_i) \leq D_0$.
\end{theorem}

In order to state the following result, Theorem \ref{tetra-compactness4},  we recall that the intrinsic flat distance between two $n$-dimensional integral current spaces $(X_i,d_i,T_i)$ is an analogue
 of the Gromov-Hausdorff distance between metric spaces. That is, it is the infimum of the flat distances of isometric orientation-preserving images of $(X_i,d_i,T_i)$, \cite{SW}. Here, an $n$-dimensional integral current space $(X,d,T)$ consists of an  $\mathcal{H}^n$-countably rectifiable metric space $(X,d)$ and an $n$-dimensional integral current structure $T$ on $\bar{X}$ as defined by Ambrosio-Kirchheim in \cite{AK} such that $\set(T)=X$. These integral currents are the generalization to metric spaces of the currents studied by Federer \cite{F} and Federer-Fleming \cite{FF}. Neither the Gromov-Hausdorff, nor the intrinsic flat convergence imply the other one, but when a sequence converges with both distances, the intrinsic flat limit is either the zero integral current space or it is contained in the Gromov-Hausdorff limit space \cite{SW}.  We note that the class of $n$-dimensional precompact integral current spaces with a uniform lower bound on their mass and the mass of their boundaries and, an upper bound on their diameter is compact under the intrinsic flat distance \cite{Wen}. 

\begin{theorem}\label{tetra-compactness4}
Let $r_0>0$, $0< \alpha <  \beta < 2$,  $C, V_0>0$ and $M_i$
be a sequence of $n$-dimensional closed oriented Riemannian manifolds that satisfy 
\[
\vol(M_i) \le V_0
\]
and the $n$-dimensional $(C,\alpha,\beta)$-tetrahedral
property for all radii $r \leq  r_0$. Then $M_i$ has a Gromov-Hausdorff and intrinsic flat convergent subsequence
whose limits agree. 
In particular, the limit is $\mathcal{H}^n$-rectifiable. 
\end{theorem}

\subsection{Comments on integral tetrahedral property versions}

In addition to the tetrahedral property, Sormani also defined the notion of $n$-dimensional $(C,\beta)$-integral tetrahedral property \cite{Sor1}. Portegies-Sormani proved all their $n$-dimensional $(C,\beta)$-tetrahedral property results using this weaker notion \cite{PorSor}. 
 
We define in  an analogous manner the $n$-dimensional $(C,\alpha, \beta)$-integral tetrahedral property and prove our results using this notion. It is easily seen that our integral property implies the previous integral property.  Nonetheless, both integral versions are equivalent up to a transformation of the various constants, see Proposition \ref{prop-EquivIntT}.

\subsection{Organization of the paper and acknowledgments}

The paper is organized as follows. In Section \ref{sec-Back} we give new examples of metric spaces satisfying the tetrahedral property. We also briefly state the basic results concerning the tetrahedral property, integral tetrahedral property, integral current spaces, intrinsic flat distance, sliced filling volume and $k$-th sliced filling.  For a thorough treatment we suggest \cite{AK}, \cite{PorSor} and \cite{Sor1}. 
Since Gromov-Hausdorff convergence is well known we just refer the reader to \cite{BurBurIva}. 

In Section \ref{sec-Cones} we analyze the tetrahedral property in Euclidean cones  $K(X)$ over metric spaces $X$ and draw some conclusions about Alexandrov spaces. First we see that if $\diam (X) \leq \pi/3$ then $K(X)$ cannot satisfy the $n$-dimensional $(C,\beta)$-tetrahedral property at its vertex, for any $n \geq 2$, $C$ and $\beta$. 
We conclude that Alexandrov spaces with nonnegative curvature do not satisfy a uniform tetrahedral property, Corollary \ref{cor-AlexNoUnif}. We show that the cone over the $2$-dimensional projective space satisfies the $3$-dimensional $(C,\beta)$-tetrahedral property.  In Lemma \ref{lem-slice-cone-tetra} we show that the slices of cones satisfying the $n$-dimensional $(C,\beta)$-tetrahedral property also satisfy the $n$-dimensional $(C,\beta)$-tetrahedral property.   We finish the section listing some of the difficulties that arise when trying to prove that Alexandrov spaces satisfy some tetrahedral property, Remarks  \ref{rmrk-AlexSpTprop}-\ref{rmrk-AlexSpTprop2}.

In Section \ref{sec-modif} we study the $(C,\alpha,\beta)$-tetrahedral property. We see that the $(C,\beta)$-tetrahedral property implies the $(C,1-\beta,1+\beta)$-tetrahedral property, Remark \ref{rmrk-modTimpT}.  The converse does not hold unless $\alpha < 1 < \beta$ as can be seen in  Example \ref{ex-planes2}, Example \ref{ex-plane&Line2} and Example \ref{ex-conesModT2}. 
We define the corresponding $(C,\alpha,\beta)$-integral tetrahedral property and show that it is equivalent to Portegies-Sormani's integral tetrahedral property up to transformation of $C, \alpha$ and $\beta$.  In this section we also prove a mass measure estimate and convergence results for integral current spaces satisfying the $(C,\alpha, \beta)$-integral tetrahedral property;  Theorem \ref{mass-ball},   Theorem \ref{tetra-compactness11} and Theorem \ref{tetra-compactness3}. From them we deduce Theorems \ref{tetra-manifold}-\ref{tetra-compactness4}.  These theorems 
are the analogues of  Theorem 3.39, Theorem 3.41 and Theorem  5.2  in \cite{PorSor} proven for the $(C,\beta)$-integral tetrahedral property. We do not claim any originality in the proofs presented in this section, they follow immediately from \cite{PorSor}.  We include them for completeness, carefully stating the hypotheses involved and noticing that their tetrahedral convergence theorems stated for manifolds, Theorem 3.41 and Theorem 5.2, also hold for integral current spaces. \\


The authors are indebted to C.~Sormani and J.~Portegies for very useful conversations and comments and to the anonymous referee for a very careful reading of the manuscript as well as several comments and suggestions which greatly improved its quality. They also wish to thank C. Sormani for the realization of Figure \ref{fig-r-big} and Figure \ref{fig-p-xy}.   This material is based in part upon work supported by the National Science Foundation under Grant No. DMS-1440140 while the second named author was in residence at the Mathematical Sciences Research Institute in Berkeley, California, during the Spring 2016 semester. The second named author also received partial support funded by C. Sormani's NSF Grant DMS - 1309360.
The first named author was supported by a UCMEXUS-CONACYT postdoctoral grant ``Alexandrov geometry''.

\section{Background}\label{sec-Back}

In this section we briefly recall the results that we use in subsequent sections. In Subsection \ref{ssec-Tprop} we give examples of spaces satisfying the tetrahedral property, state the integral tetrahedral property and go over 
results concerning the (integral) tetrahedral property for Riemannian manifolds proven by Portegies-Sormani \cite{PorSor}  and Sormani \cite{Sor1}. In Subsection \ref{ssec-currents} we provide a short introduction to integral current spaces and intrinsic flat distance following  \cite{AK} and \cite{SW}. We also state the definition of sliced filling volume and the $k$-th sliced filling. For the latter notion we state a theorem where both the intrinsic flat and Gromov-Hausdorff limits agree proved and used by Portegies-Sormani to obtain their (integral) tetrahedral convergence results \cite{PorSor}.

\subsection{Tetrahedral Property}\label{ssec-Tprop}

\subsubsection{Examples and Integral Tetrahedral Property}\label{sssec-ITprop}

To develop some intuition about the tetrahedral property we first provide two examples.  There we see that the metric spaces satisfying the tetrahedral property do not need to be Riemannian manifolds. In fact, such spaces do not need to be dimensionally homogeneous.  Given $\beta \in (0,1)$ there is $C(\beta)$ such that $(\R^2, ||\cdot||)$ satisfies the $2$-dimensional $(C(\beta), \beta)$-tetrahedral property for any radius $r>0$.  Set $C_{\R^2}(\beta)=C(\beta)$, cf. Example 2.1 in \cite{Sor1}. 


\begin{figure}
\centering
\def\svgwidth{0.75\columnwidth} 
\begingroup%
  \makeatletter%
  \providecommand\color[2][]{%
    \errmessage{(Inkscape) Color is used for the text in Inkscape, but the package 'color.sty' is not loaded}%
    \renewcommand\color[2][]{}%
  }%
  \providecommand\transparent[1]{%
    \errmessage{(Inkscape) Transparency is used (non-zero) for the text in Inkscape, but the package 'transparent.sty' is not loaded}%
    \renewcommand\transparent[1]{}%
  }%
  \providecommand\rotatebox[2]{#2}%
  \newcommand*\fsize{\dimexpr\f@size pt\relax}%
  \newcommand*\lineheight[1]{\fontsize{\fsize}{#1\fsize}\selectfont}%
  \ifx\svgwidth\undefined%
    \setlength{\unitlength}{729bp}%
    \ifx\svgscale\undefined%
      \relax%
    \else%
      \setlength{\unitlength}{\unitlength * \real{\svgscale}}%
    \fi%
  \else%
    \setlength{\unitlength}{\svgwidth}%
  \fi%
  \global\let\svgwidth\undefined%
  \global\let\svgscale\undefined%
  \makeatother%
  \begin{picture}(1,0.67283951)%
    \lineheight{1}%
    \setlength\tabcolsep{0pt}%
    \put(0,0){\includegraphics[width=\unitlength,page=1]{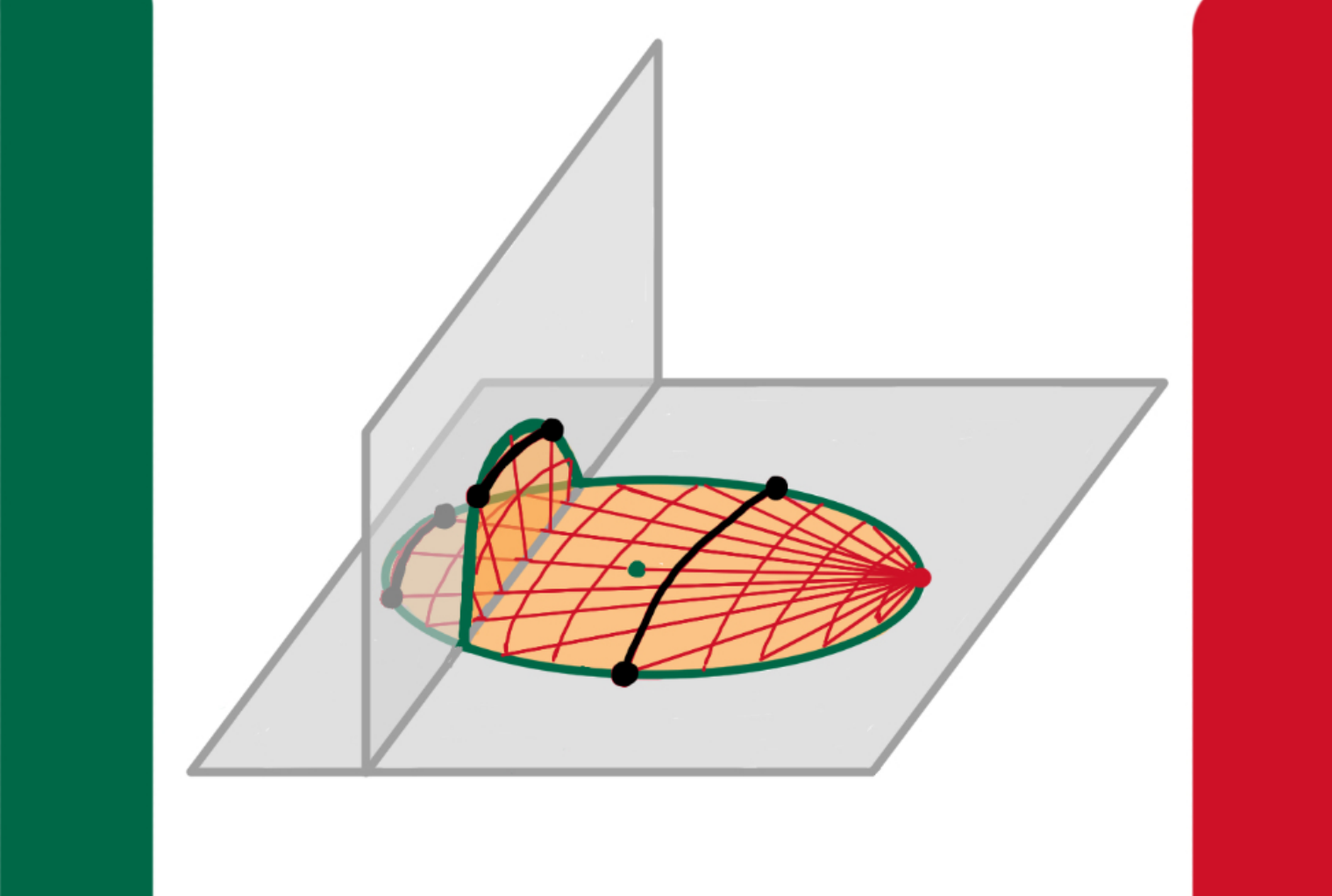}}%
    \put(0.46093637,0.27048354){\color[rgb]{0,0,0}\makebox(0,0)[lt]{\lineheight{1.25}\smash{\begin{tabular}[t]{l}$p$ \end{tabular}}}}%
    \put(0.71217202,0.24211821){\color[rgb]{0,0,0}\makebox(0,0)[lt]{\lineheight{1.25}\smash{\begin{tabular}[t]{l}$p_1$\end{tabular}}}}%
    \put(0.59465856,0.323162){\color[rgb]{0,0,0}\makebox(0,0)[lt]{\lineheight{1.25}\smash{\begin{tabular}[t]{l}$x_1$\end{tabular}}}}%
    \put(0.46600163,0.13372219){\color[rgb]{0,0,0}\makebox(0,0)[lt]{\lineheight{1.25}\smash{\begin{tabular}[t]{l}$y_1$\end{tabular}}}}%
    \put(0.42649276,0.36773606){\color[rgb]{0,0,0}\makebox(0,0)[lt]{\lineheight{1.25}\smash{\begin{tabular}[t]{l}$x_2$\end{tabular}}}}%
    \put(0.31404456,0.29884885){\color[rgb]{0,0,0}\makebox(0,0)[lt]{\lineheight{1.25}\smash{\begin{tabular}[t]{l}$y_2$\end{tabular}}}}%
    \put(0,0){\includegraphics[width=\unitlength,page=2]{two_planes.pdf}}%
    \put(0.54400622,0.49132775){\color[rgb]{0,0,0}\makebox(0,0)[lt]{\lineheight{1.25}\smash{\begin{tabular}[t]{l}$S(p_1,t_1)$\end{tabular}}}}%
  \end{picture}%
\endgroup%

\caption{Here we see the metric space $X$ described in Example \ref{ex-planes}. For $p$ in the $xy$-plane and $ r > \dist(p,y-\text{axis})$ the set $S(p,p_1;r,t_1)$ has different cardinality depending on the value of $t_1$. 
For $t_1  \in (0, \sqrt{2r^2+2r|x|})$, it  consists of two points, $x_1$ and $y_1$. For $t_1\in (\sqrt{2r^2+2r|x|},2r)$, it consists of four points, two of them belong to the $xy$-plane,  $x_2$ and $y_2$, and the other two to the upper part of the $yz$-plane,  $x_3$ and $y_3$. 
}
\label{fig-r-big}
\end{figure} 


\begin{example}\label{ex-planes}
Let $(\R^3, ||\cdot||)$ be the Euclidean metric space with the standard distance. Let $X\subset \R^3$ be the union of the $xy$-plane and the upper part of the $yz$-plane with the induced intrinsic distance, $d$. Then $(X,d)$ satisfies the $2$-dimensional $(C_{\R^2}(\beta),\beta)$-tetrahedral property for radius $r>0$ and any $\beta \in (0, \sqrt{2} -1)$.  

To prove the claim it is enough to consider points on the $xy$-plane.  Let $p=(x,y,0)$. If $x \neq 0$,  take $p_1$ to be the point on the line that passes through $(0,y,0)$ and $p$ that satisfies $||p-p_1||=r  \leq ||p_1||$.  Explicitly, $p_1=(x+rx/|x|,y,0)$.  If $x=0$, take $p_1=(r,y,0)$.  

Now the proof of the claim is divided in two cases depending on whether $r\leq \dist(p,y-\text{axis})$ or $r> \dist(p,y-\text{axis})$. 
See Figure \ref{fig-r-big} for a depiction of the second case. 
If $ r \leq \dist(p,y-\text{axis})=|x| \neq 0$,  $S(p;r)$ equals the circle of radius $r$ around $p$ in the $xy$-plane.  When $t_1 \in (0,2r)$, $S(p_1;t_1)$ equals the circle of radius $t_1$ around $p_1$ in the $xy$-plane and intersects $S(p;r)$ in exactly two points.  Hence, $(X,d)$ satisfies the $2$-dimensional $(C_{\R^2}(\beta),\beta)$-tetrahedral property at $p$ for  radius $r \leq |x|$ for any $0 < \beta < 1$.

If  $ r > \dist(p,y-\text{axis})=|x|$ then $S(p;r)$ equals the circle of radius $r$ around $p$ in the $xy$-plane union a piece of a circle in the $yz$-plane.   Hence, in the $xy$-plane $S(p_1;t_1)$ intersects $S(p;r)$ in exactly two points for any $t_1 \in (0,2r)$. In the $yz$-plane, $S(p_1;t_1)$ does not intersect $S(p;r)$ for  $t_1 < \sqrt{2r^2 + 2r|x|}$ but for $t_1 \in [ \sqrt{2r^2 + 2r|x|}, 2r)$, they intersect in exactly two points. Thus, if $t_1=\sqrt{2r^2 + 2r|x|}$  then $S(p,p_1;r,t_1)$ consists of two points and of four points if $t_1 \in (\sqrt{2r^2 + 2r|x|}, 2r)$. See Figure  \ref{fig-r-big}.  In the latter case, 
\[
\lim_{t_1^+ \to \sqrt{2r^2 + 2r|x|}} \min\{ d(x,y) | x \neq y,\,\,x,y\in S(p,p_1;r,t_1) \}=0.
\]
We see In Figure  \ref{fig-r-big} that when $t_1^+ \to \sqrt{2r^2 + 2r|x|}$ then $d(y_2,y_3) \to 0$.  
Thus, for $p$ such that $r > \dist(p,y-\text{axis})=|x|$ we restrict ourselves to $t_1 \in (0, \sqrt{2r^2 + 2r|x|})$. This implies 
that $S(p,p_1;r,t_1)$ consists only of two points and that once we choose $\beta$ appropriately then we would be able to take $C=C_{\R^2}(\beta)$.  Since $t_1$ should be bounded above by $(1+\beta)r$ we require $0  < \beta < \sqrt{2r^2 + 2r|x|} /r - 1$.
Note that $\sqrt{2}-1 \leq \sqrt{2r^2 + 2r|x|} /r - 1$.  Hence, $(X,d)$ satisfies the $2$-dimensional $(C_{\R^2}(\beta), \beta)$-tetrahedral property at $p$ for radius $r > 0$ for $0  < \beta < \sqrt{2} - 1$.
\end{example}

\begin{example}\label{ex-plane&Line}
Let $(\R^3, ||\cdot||)$ be the Euclidean metric space with the standard distance.  Let $X\subset \R^3$ be the union of the $xy$-plane and the nonnegative part of the $z$-axis with the induced intrinsic metric, $d$. Then $(X,d)$ satisfies the $2$-dimensional $(C_{\R^2}(\beta),\beta)$-tetrahedral property at $p$ in the $xy$-plane for all $r>0$ and at $p$ in the positive part of the $z$-axis only for $r > 2||p||$.

Let $p$ be contained in the $xy$-plane and $r > 0$. If $p \neq 0$, take $p_1$ to be the point in the line that passes through $0$ and $p$, and that satisfies $||p-p_1|| = r \leq ||p_1||$.  Explicitly, $p_1= p + r/||p||\,p$. If $p=0$, take any point $p_1$ in the $xy$-plane that satisfies $r=||p-p_1|| $.  If $||p||\leq r$ then  $S(p;r)$ equals the circle of radius $r$ around $p$ in the $xy$-plane union the point $z=(0,0, r-||p||)$ in the nonnegative part of the $z$-axis.  See Figure \ref{fig-p-xy}.  Hence, in the $xy$-plane $S(p_1;t_1)$ intersects $S(p;r)$ in exactly two points when $t_1 \in (0,2r)$. Now, $z \notin S(p_1;t_1)$  for $t_1 \in (0,2r)$ since $d(z, p_1)=||z||+||p_1||=r-||p|| + ||p|| +r=2r$.  If $||p|| > r$ then  $S(p;r)$ equals the circle of radius $r$ around $p$ in the $xy$-plane and does not intersect the $z$-axis. Hence, $S(p_1;t_1)$ intersects $S(p;r)$ in exactly two points when $t_1 \in (0,2r)$. Thus, the $(C_{\R^2}(\beta),\beta)$-tetrahedral property is satisfied at $p$ for $r> 0$ for any $\beta \in (0,1)$.

If $p$ is on the positive part of the $z$-axis and  $r \leq ||p||$ then  $S(p;r)$ contains only two points. Then for $p_1 \in S(p;r)$ the cardinality of $S(p,p_1;r,t)$ is less than or equal to 1. Hence,  $(X,d)$ cannot satisfy the $2$-dimensional tetrahedral property at those points with that $r$.  Suppose that $r > ||p||$ and pick $p_1$ in the $xy$-plane such that $||p_1||=r-||p||$. Then,  $S(p;r)$ equals the circle of radius $r-||p||$ around $0$ in the $xy$-plane union the point $z=p+r(0,0,1)$ on the $z$-axis. Hence, in the $xy$-plane $S(p_1;t_1)$ intersects $S(p;r)$ in exactly two points only when $t_1 \in (0,2(r-||p||))$. For $t_1 \in (0,2(r-||p||))$, $z \notin S(p_1;t_1)$ since $d(z,p_1)=||z||+ ||p_1||=2r > 2(r-||p||)$.  By definition, $t_1$ has to be contained in an interval of the form $[(1-\beta)r, (1+\beta)r]$. This means that the following inequality $(1+\beta) r < 2(r-||p||)$ must hold. Solving for $\beta$, get $\beta < 1-2||p||/r$. Moreover, $\beta$ must lie on  the interval $(0,1)$. Thus, $r > 2||p||$. 
Hence, the $(C_{\R^2}(\beta),\beta)$-tetrahedral property is satisfied at $p$ in the positive part of the $z$-axis only for $r > 2||p||$ with $\beta \in (0,1-2||p||/r)$. 
\end{example}


\begin{figure}
\centering
\def\svgwidth{0.75\columnwidth} 
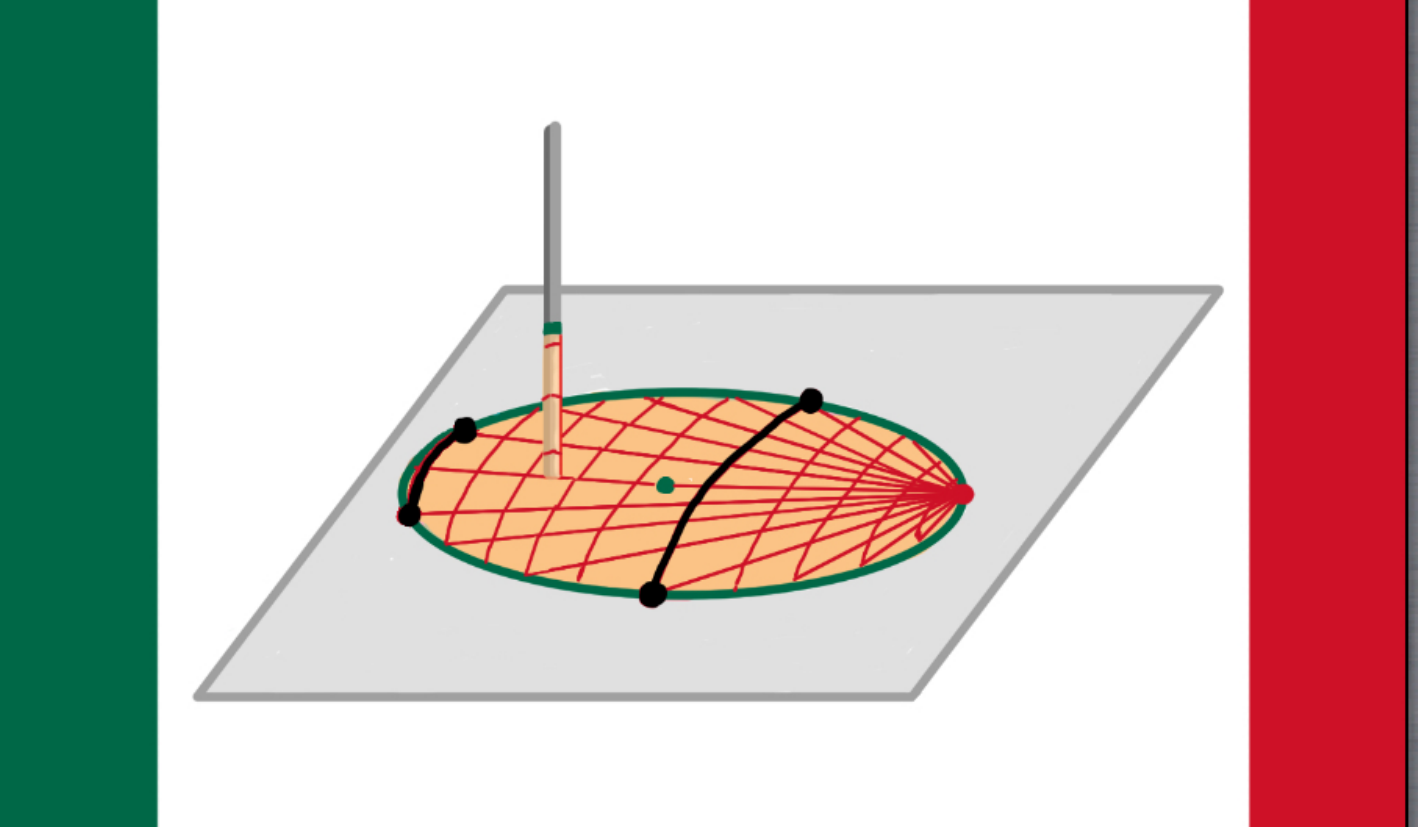
\caption{Here we see the metric space $X$ described in Example \ref{ex-plane&Line}. For $p$ in the $xy$-plane and $r \leq ||p||$ the set $S(p;r)$ consists of a circle of radius $r$ in the $xy$-plane and a point in the nonnegative part of the $z$-axis.
For $t_1  \in (0, 2(r  - ||p||))$, the set $S(p,p_1;r,t_1)$  consists of two points such as $x_1$ and $y_1$ that belong to the $xy$-plane. 
}
\label{fig-p-xy}
\end{figure} 


Now we state the integral tetrahedral property. 

\begin{definition}[Sormani]\label{def-integralCb}
Given $C > 0$ and $\beta \in (0,1)$, a metric space $(X,d)$ is said to have the $n$-dimensional  $(C,\beta)$-integral tetrahedral property at a point $p \in X$ for radius $r$ if there exist points $p_1,\dots,p_{n-1} \in \bar X$, with $d(p,p_i)=r$, such that 
\begin{equation}
\int_{(1-\beta)r}^{(1+\beta)r} \cdots \int_{(1-\beta)r}^{(1+\beta)r} h(p,r,t_1,\dots,t_{n-1}) dt_1\cdots dt_{n-1} \geq C(2\beta)^{n-1}r^n.
\end{equation}

We say that $X$ \textit{satisfies the $n$-dimensional $(C,\beta)$-integral tetrahedral property for radius r} if it satisfies the $n$-dimensional $(C,\beta)$-integral tetrahedral property at every point for radius $r$. 
\end{definition}

It is easy to see that the tetrahedral property implies the integral tetrahedral property. 

\begin{prop}[Portegies--Sormani]
Let $(X,d)$ be a metric space that satisfies the $n$-dimensional $(C,\beta)$-tetrahedral property at $p$ for radius $r$, then it satisfies the $n$-dimensional $(C,\beta)$-integral tetrahedral property at $p$ for radius $r$. 
\end{prop}

\begin{rmrk}
In fact, the following holds: Let $(X,d)$ be a metric space that satisfies the $n$-dimensional $(C,\beta)$-tetrahedral property at $p$ for radius $r$, then it satisfies the $n$-dimensional $(C',\beta')$-integral tetrahedral property at $p$ for radius $r$ for any $\beta' \in (0,1)$ with $C'=C$ if $\beta' \leq \beta$ and otherwise $C'= C(\tfrac{2\beta}{\beta'})^{n-1}$.
\end{rmrk}

\begin{example}\label{ex-integralCb}
This example is very similar to Example \ref{ex-plane&Line}, albeit slightly modified to show a space that satisfies the $2$-dimensional $(C',\beta')$-integral tetrahedral property at a point $p$ for some radius $R$ but does not satisfy the  $(C,\beta)$-tetrahedral property at the same point $p$ for the same radius $R$.    

Let $r>0$, $I= \{ (0,0,z) \mid  0\leq z \leq  r \} \subset \R^3$ and $D= \left\{(x,y,r)\mid x^2+y^2  \leq r^2\right\}$.  Let $(X,d)$ be the metric space such that $X$ consists of the union of $D$ and $I$ and $d$ is the intrinsic metric induced by $(\mathbb{R}^3, ||\cdot||)$. We now let $R=2r$, $p=(0,0,0)$ and $p_1=(x_1,y_1,r) \in \partial D$.
For $t \in (0, R)$, the set $S(p,p_1;R,t_1)$ consists of two points contained in $\partial D$. In this case, 
\[
h(p,R,t_1)= h(0, r, t_1)
\] 
where the right hand side is the function defined on the set $S(0, p'; r, t_1) \subset \R^2$  and $p'=(x_1,y_1)$.  For $t_1 \in (R, 2R)$ the set $S(p,p_1;R,t_1)$ is empty.   Therefore, it is impossible to choose $\beta \in (0,1)$ and $C>0$ such that $X$ satisfies the $2$-dimensional $(C,\beta)$-tetrahedral property at $p$ for radius $R$. 

In opposition,  for any $\beta' \in (0,1)$ we can choose $C'$ in the following manner. Integrate 
$h(p,R,t_1)= h(0, r, t_1) \geq C_{\R^2}(\beta') r$ from $(1-\beta')R$ to $(1-\beta'/2)R$ and get $C_{\R^2}(\beta') (\beta'R/2)r=   \tfrac{C_{\R^2}(\beta')}{ 2^{2}2}    (2\beta') R^2$. That is, choose $C'=C_{\R^2}(\beta')/2^{3}$. 
\end{example}

\subsubsection{Volumes of Balls and Convergence Theorems}\label{ssec-vol}

When $X$ is a Riemannian manifold that satisfies the  $(C,\beta)$-tetrahedral property at $p$ for radius $r$ the following volume estimate for balls holds. 

\begin{thm}[Portegies-Sormani, Theorem 3.39 in \cite{PorSor}]\label{thm-vol}
Let $M$ be a  closed oriented Riemannian manifold. If $M$ satisfies the $n$-dimensional $(C,\beta)$-integral tetrahedral property at a point $p$ for radius $r$, then 
\begin{equation}
\vol(B_r(p)) \geq C(2\beta)^{n-1}r^n.
\end{equation}
\end{thm}

\begin{rmrk}
Portegies-Sormani proved Theorem \ref{thm-vol} for integral current spaces as well. There, the volume is replaced by a measure coming from the current structure. Since the theory behind these spaces requires several definitions we omit the general statement but prove a similar result in Section \ref{sec-modif}, Theorem \ref{mass-ball}.
\end{rmrk}

Combining Theorem \ref{thm-vol} with Gromov-Haus\-dorff compactness theorem gives the following. 

\begin{thm}[Portegies--Sormani, Theorem 3.41 in \cite{PorSor}]\label{compGH}
\label{thm-gh}
Let $r_0 > 0$, $\beta \in (0,1)$, $C> 0$, $V > 0$ and $\{M_i\}_{i=1}^{\infty}$ be a sequence of $n$-dimensional closed Riemannian manifolds such that for all $i$ 
\begin{enumerate}
\item 
$\vol (M_i) \leq V$
\item $M_i$ satisfies the $n$-dimensional $(C,\beta)$-(integral) tetrahedral property for all  $r\leq r_0$. 
\end{enumerate}
Then  a subsequence converges in Gromov-Hausdorff sense. 
\end{thm}

Note that the metric space given in Example \ref{ex-plane&Line} does not satisfy condition $(2)$ of Theorem \ref{compGH}. There, the tetrahedral property is not satisfied at $p$ on the positive part of the $z$-axis for any $r \leq ||p||$. 

Using the intrinsic flat distance which is defined at the end of Subsection \ref{ssec-currents}, Sormani's Gromov-Hausdorff compactness theorem, Theorem \ref{compGH}, can be improved in the following way. 

\begin{thm}[Portegies--Sormani, Theorem 5.2 in \cite{PorSor}]\label{compGH=IF}
Let $r_0>0$, $0 <  \beta < 1$,  $C, V_0>0$ and $\{M_i\}$ be a sequence of $n$-dimensional closed oriented Riemannian manifolds 
such that for all $i$,

\begin{enumerate}
\item 
$\vol (M_i) \leq V$
\item $M_i$ satisfies the $n$-dimensional $(C,\beta)$-(integral) tetrahedral property for all $r \leq  r_0$. 
\end{enumerate}
Then $\{M_i\}$ has a subsequence that converges in Gromov-Hausdorff and intrinsic flat sense such that the limit spaces agree.  Hence, the limit space is $\mathcal{H}^n$-countably rectifiable.
\end{thm}

\subsection{Integral Current Spaces and Intrinsic Flat Distance}\label{ssec-currents}

In this subsection we give a brief introduction to integral current spaces and the intrinsic flat distance defined by Sormani-Wenger \cite{SW}. We recall that the definition of an integral current space is based upon work on currents in metric spaces by Ambrosio-Kirchheim \cite{AK}. We state the definitions of sliced filling volume and $k$-th sliced filling. These are the main tools used by Portegies-Sormani to prove (integral) tetrahedral results since they provide with estimates on the mass measure of balls that allows to prove Gromov-Hausdorff subconvergence and prevents intrinsic flat limits to be smaller than the Gromov-Hausdorff limit, see Theorem \ref{dist-set-2} and Theorem \ref{SF_k-compactness}. We will use these tools to derive our results.\\

An \textit{$n$-dimensional integral current space $(X, d, T)$} consists of a metric space $(X, d)$ and an $n$-dimensional integral current defined on the completion of $X$, $T\in I_n(\bar{X})$, such that $\set(T)=X$. We denote by $\mathfrak M^n$ the space of  all $n$-dimensional integral currents.  Recall that $T$ endows $\bar X$ with a finite Borel measure $||T||$,
called the mass measure of $T$ and that $\set(T)$ is defined as 
\begin{equation}
\set(T)= \left\{ x \in \bar X \, \mid \, \liminf_{r \downarrow 0} \frac{\|T\|(B_r(x))}{\omega_n r^n }> 0 \right\}.
\end{equation}
The mass of $T$ is defined as $\mass(T)=||T||(X)$.
Ambrosio-Kirchheim proved that $\set(T)$ is $\hm^n$-rectifiable. That is,  there exist Borel sets $A_i \subset \R^n$ and Lipschitz functions $\varphi_i: A_i \to X$ such that 
\begin{equation}
\hm^n\left(  \set(T) \backslash \cup_{i=1}^\infty \varphi_i(A_i) \right) = 0.
\end{equation}

An $n$-dimensional compact oriented Riemannian manifold $M$ has a canonical current given by integration of top forms:
\begin{equation}
T  (\omega) = \int_M \omega.
\end{equation}
With this current, $(X,d,T)$ is an $n$-dimensional integral current space, the mass measure of $T $ equals the Riemannian volume and $\set(T) = M$. 

Let $B_r(p) \subset X$ be a ball of radius $r$ and center $p$. To obtain Gromov-Hausdorff and intrinsic flat convergence results we are interested in calculating a lower bound for $||T||(B_r(p))$. Thus, we consider the triple 
\begin{equation}
S\left(p,r\right)=\left(\set(T\rstr B_r(p)),d,T\rstr B_r(p)\right),
\end{equation} 
where $T\rstr B_r(p)$ denotes the restriction of $T$ to $B_r(p)$.   Portegies-Sormani proved in Lemma 3.1 of \cite{PorSor} that if $\left(X,d,T\right)$ is an $n$-dimensional integral current space, then for almost every $r > 0$,
$S(p,r)$ is an $n$-dimensional integral current space.   Furthermore,
\begin{equation}
\label{ball-in-ball}
B_r(p) \subset \set(S(p,r))\subset \bar{B}_r(p)\subset X.
\end{equation} 
When $M$ is a compact oriented Riemannian manifold (with or without boundary) endowed with the canonical current then $S(p,r)$ is an integral current space for all $r > 0$ and  $\set(S(p,r))=\bar{B}_r(p)$ (see Lemma 3.2 in \cite{PorSor}). 

Lower bounds for  $\mass(S(p,r))$ are obtained by studying the sliced filling volume and the $k$-th sliced filling.  The \textit{filling volume} of  an $n$-dimensional integral current space, $M=(X,d,T)$ is defined as 
\begin{equation*}
\fillvol(M) =  \inf\{ \mass(S)\}, 
\end{equation*}
where the infimum is taken over all  $(Y,d_Y, S) \in \mathfrak M^{n+1}$ such that there exists an isometry $\varphi: X \to \set(\bdry S)$ with  $\varphi_\sharp T= \bdry S$.

Let  $F_1, F_2,...F_k: X \to \R$  be Lipschitz functions with $k\le n-1$, 
For the definition of a slice $\Slice(T, F, t)$ for $t \in \R^k$ see Portegies-Sormani, \cite{PorSor}, 
the \textit{sliced filling volume} of $\partial S(p,r)\in \intcurr_{n-1}(\bar{X})$ is defined as
\begin{equation}
\SF(p,r,F_1,...,F_k)
=\int_{t\in A_r} \fillvol(\partial\Slice(S(p,r),F,t)) \, d \mathcal{L}^k, 
\end{equation}
where $F=(F_1,...,F_k)$ and,
\begin{align*}
A_r  &= [\min F_1, \max F_1]\times [\min F_2, \max F_2]\times \cdots \times [\min F_k, \max F_k]\\
\min F_j & = \min\{F_j(x)\,|\, d(x,p) \leq r\}\\
\max F_j & =  \max\{F_j(x)\,|\, d(x,p) \leq r\}. 
\end{align*}

Given $q_1,...,q_k \in X$, set  $\rho_i(x)=d(q_i,x)$ and define
\begin{equation}
\SF(p,r,q_1,...,q_k)= \SF(p,r,\rho_1,...,\rho_k). 
\end{equation}
Then, the \textit{$k$-th sliced filling} is defined as
\begin{equation}
\SF_k(p,r)= \sup\{\SF(p,r,\rho_1,...,\rho_k)|  q_i \in X,\,\,d(p,q_i)=r\}.
\end{equation}

By bounding $\SF$ and $\SF_k$ Portegies-Sormani obtained a mass measure estimate and, a Gromov-Hausdorff and intrinsic flat convergence theorem. 

\begin{thm}[Portegies-Sormani, Theorem  3.25 in \cite{PorSor}]\label{dist-set-2}
Let $(X,d,T)$ be an $n$-di\-men\-sio\-nal integral current space and
$p_1,...,p_{n-1} \in X$. If $\bar{B}_R(p)\cap \set (\partial T)=\emptyset$ then
for almost every $r\in (0,R)$,
\begin{eqnarray*}
\mass(S(p,r)) &\ge&
\SF(p,r,p_1,...,p_{n-1}) \\
&\ge&  \int_{s_1-r}^{s_1+r} \cdots \int_{s_{n-1}-r}^{s_{n-1}+r }
h(p,r,t_1,..., t_{n-1})
\, dt_1dt_2...dt_{n-1},
\end{eqnarray*}
where $s_i=d(p_i,p)$,
\begin{equation}
h(p,r,t_1,\ldots, t_{n-1})=\left\{
\begin{array}{ll}
      \inf\{ \, d(x,y) \mid x\neq y, x,y\in S \}  & |S| \geq 2 \\
      0 & \text{otherwise} \\
\end{array} 
\right.
\end{equation}
and $S= S(p;r) \cap S(p_1; t_1)\cap \dots\cap S(p_{n-1};t_{n-1})$.
\end{thm}

\begin{thm}[Portegies-Sormani, Theorem 5.1 in \cite{PorSor}]\label{SF_k-compactness}
Let $(X_i, d_i, T_i)$ be a sequence of compact $n$-dimensional integral current spaces,
$V,A,D, C, r_0 > 0$ and some  $k\in \{0,...,n-1\}$ such that for all $i$,  $p \in X_i$ and $r \leq r_0$:
\[
\mass(T_i) \le V,\,\,\mass(\partial T_i) \le A,\,\,\diam(X_i)\le D,\,\SF_k(p,r) \ge C r^n.
\]
Then there is a subsequence of 
$(X_i,d_i,T_i)$ that converges in the Gromov-Hausdorff and
intrinsic flat senses to the same metric space. Furthermore, the limit space is $\mathcal{H}^n$-countably rectifiable. 
\end{thm}

We recall that for $S, T \in I_n(X)$, the flat distance between $S$ and $T$ in $X$ is defined as
\begin{equation*}
d_F^X(S, T)= \inf\{ \mass(U) + \mass(V) \, | \, S - T = U + \partial V, \, U \in I_n(X), \, V \in I_{n+1}(X) \}.
\end{equation*}
The intrinsic flat distance between two integral current spaces $(X, d_X, T)$ and $(Y, d_Y, S)$ is given by 
\begin{equation*}
\begin{split}
d_{\mathcal{F}}((X, d_X, T), (Y, d_Y, S)) 
&= \inf \Big\{ d_F^Z(\varphi_\# T, \psi_\# S ) \, | \, Z \text{ complete metric space}, \\
&\qquad \varphi: X \to Z, \psi: Y \to Z \text{ isometric embeddings} \Big\}.
\end{split}
\end{equation*}

We point out that $d_\mathcal{F}( (X, d_X, T), (Y, d_Y, S)) = 0$ if and only if there exists a current-preserving isometry $\varphi: X \to Y$, that is,  $\varphi$ is an isometry of metric spaces such that $\varphi_\# T = S$. 

\section{Vertices and Slices of Cones}\label{sec-Cones}

In this section we study the tetrahedral property in Euclidean cones $K(X)$ over metric spaces $X$ and in Alexandrov spaces. 
In Example \ref{ex-conesModT} we see that if $\diam (X) \leq \pi/3$ then $K(X)$ cannot satisfy the $n$-dimensional $(C,\beta)$-tetrahedral property at its vertex, $o$, for any $n \geq 2$, $C$ and $\beta$.  Furthermore, $K(X)$ cannot satisfy the $n$-dimensional $(C,\beta)$-integral tetrahedral property at $o$ for any $n \geq 2$, $C$ and $\beta \in (0, 1-  \sqrt{2 (1-\cos(\diam(X) ))})$.   From this we conclude in Corollary \ref{cor-AlexNoUnif} that Alexandrov spaces  with nonnegative curvature do not satisfy a uniform tetrahedral property.  

In  Example \ref{ex-projective} we show that the cone over the $2$-dimensional projective space satisfies the $3$-dimensional $(C,\beta)$-tetrahedral property. Thus, the existence of topological singularities is not per se an obstruction. Note that with the usual round metric, $\diam (\mathbb{R}P^2)=\pi/2 > \pi/3$.  In Lemma \ref{lem-slice-cone-tetra} we show that the slices of cones satisfying the $n$-dimensional tetrahedral property also satisfy the $n$-dimensional tetrahedral property.  We finish this section stating a failed attempt to show that 
Alexandrov spaces satisfy some tetrahedral property at regular points.

\begin{definition}[cf. Definition 3.6.12 in \cite{BurBurIva}]\label{cone} 
Let $(Y,d_{Y})$ be a metric space with $\diam(Y) \leq \pi$. Then $K(Y):= Y\times [0,\infty) / Y\times \{0\}$ endowed with the distance
\[
d_{K}((x,t),(y,s)):=\sqrt{t^2+s^2-2st\cos d_{Y}(x,y)}.
\]
is called the Euclidean cone over $Y$. The vertex of $Y$, that is the point corresponding to $Y \times \{0\}$, is denoted by $o$.
\end{definition}

\begin{example}\label{ex-conesModT}
Let $(X,d)$ be a metric space with $\diam(X) < \pi/3$.  Then for any integer $n \geq 2$, $C>0, \beta \in (0,1)$ and $r >0$, the cone  $K(X)$ over $X$ does not satisfy the $n$-dimensional $(C,\beta)$-tetrahedral property at its vertex $o$ for radius $r$.  This follows from the following calculation 
\begin{equation}\label{eq-ConeD}
d^2_K((x,r),(y,r))= 2r^2 - 2r^2 \cos(d(x,y))= 2r^2 (1-\cos(d(x,y))) <  r^2,
\end{equation}
which shows that for any $(x,r) \in S(o ; r)= X \times{\{r\}}$, $S((x,r) ; r) = \emptyset$. Therefore, for any $(x_1,r),...,(x_{n-1},r)  \in S(o; r)$ we have
$h(o,r,r,...,r)=0$.

We can easily see that (\ref{eq-ConeD}) also shows that for any $n \geq N$, $C>0$ and  $\beta \in (0, 1-  \sqrt{2 (1-\cos(\diam(X) ))})$, the cone  $K(X)$ over $X$ does not satisfy the $n$-dimensional $(C,\beta)$-integral tetrahedral property at its vertex $o$, for radius $r >0$.  Since,
\begin{align*}
d_K((x,r),(y,r)) \leq &  r \sqrt{2 (1-\cos(  \diam(X))))} < (1-\beta)r.
\end{align*}
which shows that  for any $(x,r) \in S(o ; r)= X \times{\{r\}}$, $S((x,r) ; t_1) = \emptyset$ for any $t_1 \in [(1-\beta)r, (1+\beta) r]$. Thus, for any choice of $(x_1,r),...,(x_{n-1},r)\in S(o; r)$ and $t_i \in [(1-\beta)r, (1+\beta)r]$, we have $h(o,r,t_1,...,t_{n-1})=0$.
\end{example}

\begin{rmrk}\label{rmrk-diam}
The tetrahedral property gives a bound from below of the diameter of a space as we saw in the previous example. Indeed, 
if a space satisfies the tetrahedral property at a point for radius $r$ then the diameter of the space is bigger than $r$. 
\end{rmrk}

\begin{cor}\label{cor-AlexNoUnif}
The class of $n$-dimensional Alexandrov spaces  with nonnegative curvature do not satisfy 
any uniform $n$-dimensional $(C,\beta)$-tetrahedral property for any radius $r >0$. 
Even if we require the diameter  of the spaces to be bigger than $r$ or if we require the spaces to be closed, i.e. to be complete and with no boundary.
\end{cor}

\begin{proof}
Let $r>0$.  If we do not require a diameter bound, then any Alexandrov space in this class with diameter less than $r$ does not satisfy the tetrahedral property at any point 
for radius $r$.  If we require the spaces to have diameter bigger than $r$, then  recall that in Theorem 10.2.3 of \cite{BurBurIva}  it is shown that for any complete Alexandrov space $(Y,d)$ with curvature $\geq 1$, the Euclidean cone $K(Y)$ is an Alexandrov space with curvature  $\geq 0$. Thus, by Example \ref{ex-conesModT} we see that all Euclidean cones $K(Y)$, with $\diam(Y) < \pi/3$ do not satisfy the tetrahedral property for radius $r$.
If the spaces are additionally required to be closed, we can choose appropriate thin flat tori such that the conclusion holds. We could also choose spaces $S$ homeomorphic to the sphere that look like ellipsoids, that is, like a long cylinder with two spherical caps, one at each end.  If we want to allow some negative curvature, one can construct dumbbells $(X,d)$ with thin necks and choose $p \in X$ such that  $S(p;r)$ is contained in the neck and $\diam(S(p;r),d)< r$. This will imply that $X$ does not satisfy the tetrahedral property at $p$  for radius $r$. 
\end{proof}

\begin{cor}
The class of $n$-dimensional Riemannian manifolds with curvature bounded below by $\kappa$ do not satisfy 
any uniform $n$-dimensional $(C,\beta)$-tetrahedral property for any radius $r$.  
Even if we require their diameter to be bigger than $r$ or if we require the spaces to be closed, i.e. to be complete and with no boundary.
\end{cor}

\begin{proof}
Except for $K(Y)$, all the spaces considered in the proof of  Corollary \ref{cor-AlexNoUnif} provide a proof of the statement. 
For a space similar to $K(Y)$, with nonnegative curvature, but which is a Riemannian manifold, we can consider $n$-dimensional paraboloids.  
\end{proof}

In general, the existence of \textit{topologically singular points}, in the sense of Alexandrov geometry, is not a priori an obstruction to the tetrahedral property as the following example shows. Let us recall that a point $p$ in an $n$-dimensional Alexandrov space is topologically singular if the space of directions at $p$ is not homeomorphic to an $(n-1)$-dimensional sphere. We refer the reader to \cite{BurBurIva}.  

\begin{example}\label{ex-projective} 
Let us consider $K(\mathbb{R}P^2)$ and let $o$ be the vertex. We claim that there exist $r>0$ and constants $C=C(o)>0$, $\beta=\beta(o)\in(0,1)$ such that $X$ satisfies the $3$-dimensional $(C,\beta)$-tetrahedral property  at $o$ for radius $r$.

There exists an isometric involution $\iota$ on a $3$-ball $B_R(0)\subset \mathbb{R}^3$, such that $B_R(0)/\iota$ is isometric to $K(\mathbb{R}P^2)$. The involution is given as the conification of the action induced by the antipodal map on $\mathbb{S}^2$. Recall that in Example 2.1 in \cite{Sor1} it was shown that $\R^3$ satisfies the $3$-dimensional $(C_{\R^3}(\beta), \beta)$-tetrahedral property for any radius $r>0$ and $\beta \in (0,1)$.  Let $r>0$ and  $p_1, p_2$ be the points coming from that example. It is clear that $p_i$ and the set  $S(o, p_1, p_2;r, s_1, s_2)$ are contained in a fundamental domain of the antipodal map action. Furthermore, the position of the $p_i$ and the radii $s_i$ can be perturbed so that 
\[
\pi\left(S(o, p_1, p_2;r, s_1, s_2)\right) = S(\pi(o), \pi(p_1), \pi(p_2);r, s_1,s_2).
\]
where $\pi:\mathbb{R}^3\to K(\mathbb{R}P^2)$ is the canonical projection.  Then the $3$-dimensional $(C_{\mathbb{R}^3}(\beta),\beta)$-tetrahedral property is satisfied for radius $r$ at $o$. 
\end{example}

\begin{lem}\label{lem-slice-cone-tetra} 
Let $(X,d)$ be a metric space with $\diam(X)\leq \pi$ satisfying the $n$-dimensional $(C,\beta)$-tetrahedral property for radius $r \leq \pi/2$. Let  $K(X)$ be the cone of $X$ with the cone distance $d_{K}$. Then for $s >0$, the slice $X\times\{s\}\subset K(X)$  with the restriction of $d_{K}$ satisfies the $n$-dimensional $(C_r,\beta_r)$-tetrahedral property for radius $r'$, where
\[
 C_r=\sqrt{\frac{1-\cos(Cr)}{1-\cos(r)}},
 \]
 \[
 0< \beta_r \leq \max\left\{ 1-\sqrt{\frac{1-\cos((1-\beta)r)}{1-\cos(r)}},  \sqrt{\frac{1-\cos((1+\beta)r)}{1-\cos(r)}}-1 \right\}
  \] and 
\[
r'= \sqrt{2s^2(1-\cos(r))}.
\]
\end{lem}

\begin{proof}
We begin by considering a point $p\in X$ and noting that, since $X$ satisfies the $n$-dimensional $(C,\beta)$-tetrahedral property for radius $r$ at every point, there exist points $p_1,\ldots, p_{n-1}\in S(p;r)$ such that for $t_i\in [(1-\beta)r,(1+\beta)r]$,
\begin{equation}\label{eq-none}
S(p,p_1, ..., p_{n-1}; r, t_1,...,t_{n-1})\neq \emptyset.
\end{equation}

We claim that if $t_i\in [(1-\beta)r,(1+\beta)r]$, then for $t_i'=  \sqrt{2s^2(1-\cos(t_i))}$,
\begin{equation*}
S( (p,s),(p_1,s), ..., (p_{n-1},s); r', t'_1,...,t'_{n-1}) = S(p,p_1, ..., p_{n-1}; r, t_1,...,t_{n-1}) \times \{s\}.
\end{equation*}
We recall that the term on the left hand side of the equality stands for the intersection of the metric spheres 
$S((p,s); r')$, $S((p_i,s);  t'_i)$, $i=1,...n-1$,  contained in $K(X)$, while the term $S(p,p_1, ..., p_{n-1}; r, t_1,...,t_{n-1})$ appearing on the right hand side stands for the intersection of metric spheres $S(p; r')$, $S(p_i;  t_i)$, $i=1,...n-1$,  contained in $X$. 

In particular, from equation (\ref{eq-none}) and the claim it will follow that $$S( (p,s),(p_1,s), ..., (p_{n-1},s); r', t'_1,...,t'_{n-1})\neq \emptyset.$$
Since the cosine function is decreasing in $[0,\pi]$ and the square root is increasing in $[0,\infty)$ this would imply that for $(x,s)\neq(y,s)$ such that
$$(x,s),(y,s)\in S( (p,s),(p_1,s), ..., (p_{n-1},s); r', t'_1,...,t'_{n-1})$$
then   
\begin{eqnarray*}
d_{K}((x,s),(y,s))& = & \sqrt{2s^2\left(1-\cos(d(x,y)) \right)} \\ 
& \geq & \sqrt{2s^2(1-\cos(Cr))}= C_{r}r'.
\end{eqnarray*}

\bigskip 
To prove the claim, take $x\in S(p,p_1, ..., p_{n-1}; r, t_1,...,t_{n-1})$. Then, 
\begin{equation}\label{eqdKp}
d_{K}((p,s),(x,s))= \sqrt{2s^2-2s^2\cos(d(p,x))}=r'
\end{equation}
\begin{eqnarray}\label{eqdKpi}
d_{K}((p_i,s),(x,s)) & = & \sqrt{2s^2(1-\cos(d(p_i,x)))} \\ 
& = & \sqrt{2s^2(1-\cos(t_i))}. 
\end{eqnarray}
Note that since $t_i\in [(1-\beta)r,(1+\beta)r]$, 
\[
\sqrt{2s^2(1-\cos((1-\beta)r))}\leq t'_i \leq \sqrt{2s^2(1-\cos((1+\beta)r))}.    
\]

From (\ref{eqdKp}) and (\ref{eqdKpi}),
\[
S( (p,s),(p_1,s), ..., (p_{n-1},s); r', t'_1,...,t'_{n-1}) \supset   S(p,p_1, ..., p_{n-1}; r, t_1,...,t_{n-1}) \times \{s\}. 
\] 
The proof of the claim concludes noticing that the map $\Phi: [(1-\beta)r, (1+\beta)r] \to  \R$  given by 
\[
\Phi(t'):= \sqrt{2s^2(1-\cos(t))}.
\]
is a bijection onto its image. 

The lemma follows from the following calculations:  
We know that for $t_i\in [(1-\beta)r,(1+\beta)r]$, 
\[
\sqrt{2s^2(1-\cos((1-\beta)r))}\leq t'_i \leq \sqrt{2s^2(1-\cos((1+\beta)r))}.    
\]
From the definition of $\beta_r$ and $r'$,
\[
\sqrt{2s^2(1-\cos((1-\beta)r))} =\sqrt{\frac{1-\cos((1-\beta)r)}{1-\cos(r)}} \,r'   \geq (1- \beta_r) \,r' .
\]
and 
\[
 \sqrt{2s^2(1-\cos((1+\beta)r))} = \sqrt{\frac{1-\cos((1+\beta)r)}{1-\cos(r)}} \, r' \leq (1+ \beta_r) \,r' .
\]
%
%
\end{proof}

We stress that the previous lemma states that for each $s>0$, the slice $X\times\{s\}\subset K(X)$  with the restriction of $d_{K}$ satisfies the $n$-dimensional $(C,\beta)$-tetrahedral property for some radius,  but we do not claim that  $K(X)$ satisfies the $(n+1)$-dimensional $(C,\beta)$-tetrahedral property.  

\bigskip
It was noted in \cite{PorSor} that there is no uniform tetrahedral property on manifolds with positive scalar curvature even when the volume of the balls is uniformly bounded below by the volume of Euclidean balls. It was conjectured that one might have the $(C, 1/2)$-tetrahedral property or an integral version on manifolds with uniform sectional curvature lower bounds (and even with uniform Ricci curvature lower bounds) and a uniform volume lower bound.  It was observed that two of the difficulties arising when trying to prove that a metric space satisfies the tetrahedral property at a point $p$ for a given radius $r>0$ are the following; even if one knows that $h(p, r, ..., r) > Cr$ the triangle inequality does not ensure that  $h(p, r, t_1, ..., t_{n-1}) \approx Cr$ for $t_i$ close to $r$, nor that $S(p, p_1,...,p_{n-1}; r, t_1 , ..., t_{n-1})=\emptyset$. In the former case and even for Riemannian manifolds, there is the possibility that there exist points $x_1,x_2, y_1,y_2  \in S(p,p_1,..., p_{n-1};r, t_1, ..., t_{n-1})$ such that $d(x_i,y_j) \approx Cr$, $i,j=1,2$, but $d(x_1,x_2), d(y_1,y_2) \approx 0$. In the following remarks we consider Alexandrov spaces and go over two particular approaches to prove that the space satisfies the tetrahedral property at a point for some radius. In the first one we are unable to choose points $p_i$, $i=1,...,n-1$, and in the second we cannot  rule out the existence of points $x_i,y_i$ such that $d(x_1,x_2), d(y_1,y_2) \approx 0$.

\begin{rmrk}\label{rmrk-AlexSpTprop} 
It is known that if $(X,d)$ is an $n$-dimensional Alexandrov space and $\{p_i, q_i\}_{i=1}^{n}$ an $(n,\delta)$-strainer for $p \in X$ with $\delta \leq \tfrac 1{100n}$ then there exists an open set $U \subset X$ such that $p\in U$ and, the function $f:  X  \to \R$ given by $f(x)= (d(x,p_1), ..., d(x,p_n))$ for any $x \in X$ is a bilipschitz function into its image when restricted to $U$ with Lipschitz constant bounded by $\sqrt{n}$ and the Lipschitz constant of its inverse bounded by $500n$.  See  Proposition 10.8.15 in \cite{BurBurIva}.  Knowing that $f|_U: U  \to f(U)$ is bilipschitz suggests we could try to show that if $d(p,p_i)=r$, $i=1,...,n-1$, the set $S(p,p_1,...,p_{n-1} ;  r, t_1,...,t_{n-1})$ and the function $h(p,r,t_1,\ldots, t_{n-1})$
behave in a suitable way so that $(X,d)$ satisfies the tetrahedral property at $p$ for radius $r$ for appropriately chosen $0 <\beta <1$. In spite of that, in order to prove that $f|_U: U \to f(U)$ is bilipschitz one requires that the distance from each $p_i$ to $U$ is sufficiently big.  Hence, $S(p; r) \cap U = \emptyset$ and in particular
$S(p,p_1,...,p_{n-1} ;  r, t_1,...,t_{n-1}) \subset S(p; r)$ does not intersect $U$. Thus, one is unable to use the bilipschitz map.  
\end{rmrk}

\begin{rmrk}\label{rmrk-AlexSpTprop2}
Let $(X,d)$ be an $n$-dimensional Alexandrov space and $p\in X$. Assume that there exists an $(n,\delta)$-strainer $\{p_i, q_i\}_{i=1}^{n}$ such that there is an open set $U \subset X$ with $p\in U$ and a bilipschitz map $f|_{U}: U \to  f(U) \subset \R^n$ with
$$
(1+\vare)^{-1} d(x,z) \leq ||f(x) - f(z) || \leq (1+\vare)d(x,z),
$$
where $f(x)= (d(x,p_1), ..., d(x,p_n))$ for any $x \in X$. That is, we can assume that $p$ is a regular point of $X$.  Let's attempt to prove that $(X,d)$ satisfies the $n$-dimensional $(C, \beta)$-tetrahedral property at $p$ for some radius $r$.  Assume without loss of generality that $f(p)=0$ and let $r>0$ be small enough so that $B_{2r}(p)\subset U$. Since the $n$-dimensional Euclidean space satisfies the $n$-dimensional $(C_{\R^n}(\beta), \beta)$-tetrahedral property at $0$ for any radii $r>0$ and any $\beta \in (0,1)$, there exist points $a_1,...,a_{n-1} \in \R^n$,  $||a_i||=r$, as the ones appearing in Definition \ref{def-Tprop}.  Then for $\vare>0$ sufficiently small and using the fact that $f|_U$ is $(1+\vare)$-bilipschitz into its image we can find points $p_i \in X$ such that $d(p,p_i)=r$ and $||a_i -  f(p_i)||$ are close, $i=1,...,n-1$. Then one can show that $f(S(p,p_1,..., p_{n-1};r,t_1,...,t_{n-1}))$ is Hausdorff close to $S(0,a_1,..., a_{n-1};r,t_1,...,t_{n-1})$. The closeness of the points and the Hausdorff closeness can be given in terms of $\vare$.  The Hausdorff closeness ensures  
the existence of at least two points in $S(p,p_1,..., p_{n-1};r,t_1,...,t_{n-1})$, as in the Euclidean case, whose distances are bounded below by a number close to $C_{\R^n}(\beta) r$. But we are unable to rule out the possibility of more points appearing.  If one tries to overcome the problem by taking $\vare\to 0$ in order to get bilipschitz functions $f_\vare:  U_\vare  \to  f_\vare(U_\vare)$ with constants approaching $1$, then the previous sets are closer to each other in Hausdorff distance. Though, $\diam(U_\vare)\to 0$ and we would have a sequence of radii $r_\vare  \to 0$.  We conjecture that if the balls in $U$ are convex the present issue could be solved. 
\end{rmrk}

\section{$(C,\alpha,\beta)$-Tetrahedral Property}\label{sec-modif}

In this section we recall the definition of the $(C,\alpha,\beta)$-tetrahedral property, Definition \ref{def-modTprop}, present the definition of the  $(C,\alpha,\beta)$-integral tetrahedral property, provide examples of spaces that satisfy these properties and prove that the main results of the $(C,\beta)$-tetrahedral property such as those related to the Gromov-Hausdorff and intrinsic flat convergences hold,
Theorems  \ref{tetra-manifold}-\ref{tetra-compactness4}.

In Subsection \ref{ssec-abTp} we see that the $(C,\beta)$-tetrahedral property implies the $(C,1-\beta,1+\beta)$-tetrahedral property, Remark \ref{rmrk-modTimpT}. Then we analyze some examples presented above. 
These examples also evidence the improvement over the $(C,\beta)$-tetrahedral property: The $(C,\alpha,\beta)$-tetrahedral property may be satisfied in spaces in which the $(C,\beta)$-tetrahedral property fails.   In  Example \ref{ex-planes2} we present a metric space $(X,d)$ that satisfies the $(C, \beta)$-tetrahedral property at $p$ for $r>0$ and $\beta <c(p,r)$, where $c(p,r)$ denotes a function that depends on $p$ and $r$. Hence, $(X,d)$ satisfies the $(C, 1-\beta, 1+\beta)$-tetrahedral property. Moreover, $(X,d)$ satisfies the $(C, \alpha, \beta)$ property for $c(p,r) < \alpha < \beta <2$.  In Example \ref{ex-plane&Line2} we show a metric space $(X,d)$ that satisfies the $(C, \alpha, \beta)$-tetrahedral  property at some points $p$ only for $r > ||p||$. Meanwhile, it satisfies the $(C, \beta)$-tetrahedral  property at those points only for $r > 2||p||$.   Recall that Example \ref{ex-conesModT} deals with the Euclidean cone of spaces of small diameter. We showed that these cones do not satisfy the $(C,\beta)$-tetrahedral property at their vertex.  In  Example \ref{ex-conesModT2} we see that the Euclidean cone of spheres with small diameter do satisfy the $(C, \alpha, \beta)$-tetrahedral property for adequate choices of the parameters.  We comment that it is not straightforward to prove that 
Alexandrov spaces satisfy the $(C, \alpha, \beta)$-tetrahedral property at a point for certain radius by considering strainers
since Remarks  \ref{rmrk-AlexSpTprop}- \ref{rmrk-AlexSpTprop2}  also apply in this case.

 In Subsection \ref{ssec-abiTp} we define the $(C,\alpha,\beta)$-integral tetrahedral property, provide an example and show that our generalized  $(C,\alpha,\beta)$-integral tetrahedral property is equivalent to the $(C',\beta')$-integral tetrahedral property, Proposition \ref{prop-EquivIntT}.  We recall that Portegies-Sormani  proved their tetrahedral results using their integral tetrahedral property and that our proofs 
are almost identical to theirs.

For simplicity, in Subsection \ref{ssec-vol}, we stated Portegies-Sormani's volume estimate, Theorem \ref{thm-vol},
for Riemannian manifolds satisfying the $(C,\beta)$-tetrahedral property. We remark that this result also holds for integral current spaces  (Theorem 3.39  in \cite{PorSor}).
In Subsection \ref{ssec-mass} we prove the analogous volume estimate for integral current spaces satisfying the $(C,\alpha, \beta)$-integral tetrahedral property;  Theorem \ref{mass-ball}.  As a corollary we get Theorem \ref{tetra-manifold}.

In Subsection \ref{ssec-convergence} we prove convergence results for integral current spaces satisfying the $(C,\alpha, \beta)$-integral tetrahedral property;  Theorem \ref{tetra-compactness11} and Theorem \ref{tetra-compactness3}.  As corollaries we get the convergence theorems for manifolds stated in the introduction, that is, Theorem \ref{tetra-compactness2}  and  Theorem \ref{tetra-compactness4}. These theorems for manifolds are analogous to Portegies-Sormani's results,  Theorem \ref{compGH} and Theorem \ref{compGH=IF}.  We do not claim any originality in the proofs presented in these two last subsections. They follow immediately from Portegies-Sormani's proofs. Though we stated the results keeping track of all the hypotheses needed and state the convergence theorems not just for manifolds.


\subsection{$(C,\alpha,\beta)$-Tetrahedral Property and Examples}\label{ssec-abTp}

Let $(X,d )$ be a metric space, 
recall that $S(p;r)=\{ x \in X | \,d(x,p)=r \}$ and that 
$$S(x_1,\dots,x_j;t_1,\dots,t_j) = \bigcap_{i=1}^j   S(x_i;t_i).$$
Furthermore, the metric completion of $X$ is denoted by $\bar X$ and the cardinality of a set $S$ by $|S|$.

\begin{definition}[$(C,\alpha,\beta)$-tetrahedral property]
Let $C>0$ and $\alpha, \beta\in (0,2)$, $\alpha < \beta$.  A metric space $(X,d)$ satisfies the $n$-dimensional $(C,\alpha,\beta)$-tetrahedral property at a point $p$ for radius $r$ if there exist points $p_1,..., p_{n-1}\in \bar{X}$ such that $d(p,p_i)=r$ and for all $(t_1,..., t_{n-1})\in [\alpha r,\beta r]^{n-1}$ the following holds  
\[ 
h(p,r,t_1,\ldots, t_{n-1})\geq Cr,
\]
where 
\begin{equation}
h(p,r,t_1,\ldots, t_{n-1})=\left\{
\begin{array}{ll}
      \inf\{ \, d(x,y) \mid x\neq y, x,y\in S \}  & |S| \geq 2 \\
      0 & \text{otherwise} \\
\end{array} 
\right.
\end{equation}
and $S= S(p,p_1,\dots,p_{n-1};r,t_1,\dots,t_{n-1})$.

We say that $X$ \textit{satisfies the $n$-dimensional $(C,\alpha,\beta)$-tetrahedral property for radius r} if it satisfies the  $n$-dimensional $(C,\alpha, \beta)$-tetrahedral property at every point for radius $r$. 
\end{definition}

\begin{rmrk}\label{rmrk-modTimpT}
Let $(X,d)$ be a metric space that satisfies the $n$-dimensional $(C,\beta)$-tetrahedral property at $p$ for radius $r$. By definition, $(X,d)$ satisfies the $n$-dimensional $(C, 1-\beta, 1+\beta)$-tetrahedral property at $p$ for radius $r$.  
If $(X,d)$ satisfies the $n$-dimensional $(C, \alpha, \beta)$-tetrahedral property at $p$ for radius $r$ then it satisfies 
the $n$-dimensional $(C, \min\{1-\alpha, \beta - 1\})$-tetrahedral property at $p$ for radius $r$ only if $\alpha < 1 < \beta$.
\end{rmrk}

\begin{example}\label{ex-planes2} 
Recall Example \ref{ex-planes} where $X\subset \R^3$ equals the union of the $xy$-plane and the upper part of the $yz$-plane with the induced intrinsic distance, $d$.  
We showed that $(X,d)$ satisfies the $2$-dimensional $(C_{\R^2}(\beta), \beta)$-tetrahedral property at $p \in X$ for radius $r \leq \max\{\dist(p,xy-\text{plane}), \dist(p,yz-\text{plane})\}$ for any $0 <  \beta < 1$.  We also showed that  $(X,d)$ satisfies the $2$-dimensional $(C_{\R^2}(\beta), \beta)$-tetrahedral property at $p \in X$ for radius $r  > \max\{\dist(p,xy-\text{plane}), \dist(p,yz-\text{plane})\}$ for $0  < \beta < \sqrt{2r^2 + 2r|x|} /r - 1$.  

With the new definition, $(X,d)$ satisfies the $2$-dimensional $(C, \alpha, \beta)$-te\-tra\-he\-dral property at $p \in X$ for radius $r \leq \max\{\dist(p,xy-\text{plane}), \dist(p,yz-\text{plane})\}$ for any $0 < \alpha < \beta < 2$ and, at $p \in X$ for radius $r  > \max\{\dist(p,xy-\text{plane}), \dist(p,yz-\text{plane}\}$ for $0  < \alpha < \beta < \sqrt{2r^2 + 2r|x|} /r $ and $\sqrt{2r^2 + 2r|x|} /r < \alpha < \beta < 2$. 
\end{example}

\begin{example}\label{ex-plane&Line2}
Recall Example \ref{ex-plane&Line} in which $X\subset \R^3$ consisted of the union of the $xy$-plane and the nonnegative part of the $z$-axis with the induced intrinsic distance, $d$.  We showed that $X$ satisfies the $2$-dimensional $(C_{\R^2}(\beta),\beta)$-tetrahedral property at $p \in X$ contained in the $xy$-plane for all $r > 0$. We also showed that $X$ satisfies the $2$-dimensional $(C_{\R^2}(\beta),\beta)$-tetrahedral property at $p$ on the positive part of the $z$-axis  only for $r > 2||p||$, where $\beta \in (0, 1-2||p||/r)$. Here we prove that $X$ satisfies the $2$-dimensional $(C, \alpha,\beta)$-tetrahedral property at points $p$ on the positive part of the $z$-axis for $r  \in (||p||, 2||p||]$ but not for $r  \in (0, ||p||]$.

If $p$ is in the positive part of the $z$-axis,  take $r  \in (||p||, 2||p||]$ and
pick $p_1$ in the $xy$-plane such that $||p_1||=r-||p||$. Note that  $S(p;r)$ equals the circle of radius $r-||p||$ around $0$ in the $xy$-plane union the point $z=p+r(0,0,1)$ on the $z$-axis. Hence, when $t_1 \in (0,2(r-||p||))$  $S(p_1;t_1)$ intersects $S(p;r)$ in exactly two points in the $xy$-plane. For $t_1 \in (0,2(r-||p||))$, $z \notin S(p_1;t_1)$ since $d(z,p_1)=||z||+ ||p_1|| = 2r >  2(r-||p||)$ and $r  \in (||p||, 2||p||]$ implies $2(r-||p||) \leq r$.  Thus, $(X,d)$ satisfies the $2$-dimensional $(C,\alpha, \beta)$-tetrahedral property at $p$ for $r \in (||p||, 2||p||] $ and $0 < \alpha < \beta < 2(r-||p||)/r$. 

If $p$ is in the positive part of the $z$-axis and  $r \leq ||p||$ then  $S(p;r)$ contains only two points. Then for $p_1 \in S(p;r)$ the cardinality of $S(p,p_1;r,t)$ is less than or equal to 1. Hence,  $(X,d)$ cannot satisfy the $2$-dimensional $(C, \alpha,\beta)$-tetrahedral property at those points with that $r$.
\end{example}

\begin{example}\label{ex-conesModT2}
 Let $\mathbb{S}^2(r)$ be the sphere of radius $r\leq 1$ (so that its diameter is less than or equal to $\pi$) and consider $K(\mathbb{S}^2(r))$ with the cone metric. As pointed out in Example \ref{ex-conesModT},  if $r \leq 1/3$ then the cone $K(\mathbb{S}^2(r))$ does not satisfy the $3$-dimensional $(C,\beta)$-tetrahedral property at its vertex for any radius. We show that it does satisfy the $3$-dimensional $(C,\alpha, \beta)$-tetrahedral property at its vertex, $o$, for any $t>0$ and some $\alpha < \beta$ and $C$. 
 
Recall that $S(o;t)=\mathbb{S}^2(r) \times \{t\}$.  Let $p_1,p_2 \in \mathbb{S}^2(r)$ be two distinct points. 
From the formula $d_K((p,t), (x,t))= 2t\sin(\tfrac12 d(p,x))$  we see that for $t_i' =  2 \arcsin (\tfrac{t_i}{2t})$ the following holds 
\begin{equation*}
S(o, (p_1,t), (p_2,t);t,t_1, t_2) = S(p_1,p_2; t'_1,t'_2) \times \{t\}.
\end{equation*}

It can be seen that the set $S(p_1,p_2; t'_1,t'_2)$ has cardinality two when 
$$t'_1,t'_2 \in [\alpha', \beta'] \subset (\tfrac 12 d(p_1,p_2), \min\{ \tfrac 32 d(p_1,p_2),   \pi r - \tfrac 12d(p_1,p_2)\}).$$
Furthermore, the function $(t'_1,t'_2) \in [\alpha',\beta']^2 \mapsto d(x ,y)$, where $x \neq y \in S(p_1,p_2; t'_1,t'_2)$, 
attains its minimum, $d(x_m,y_m)$, at one of the points $$\{(\alpha',\alpha'),(\alpha',\beta'),(\beta',\alpha'),(\beta',\beta')\}.$$  
Hence, from the formula $d_K((p,t), (x,t))= 2t\sin(\tfrac12 d(p,x))$, if
\begin{equation*}
\alpha >   2\sin( \tfrac 14 d(p_1,p_2) )   ,
\end{equation*} 
\begin{equation*}
\beta <  2  \min\{   \sin(\tfrac 34 d(p_1,p_2) )  ,   \sin( \tfrac12 ( \pi r -  \tfrac 12 d(p_1,p_2))) \}
\end{equation*} 
and 
\begin{equation*}
 C= 2 \sin (\tfrac 12 d(x_m,y_m)), 
\end{equation*} 
where $(x_m,y_m)$ is the minimum of the function $(t'_1,t'_2) \in [\alpha',\beta']^2 \mapsto d(x ,y)$ with $\alpha'= 2\arcsin(\alpha/2)$
and $\beta'= 2\arcsin(\beta/2)$, then 
the metric space $K(\mathbb{S}^2(r))$ satisfies the $3$-dimensional $(C,\alpha, \beta)$-tetrahedral property at  $o$ for radius $t$
with points $(p_1,t)$ and $(p_2,t)$.   
\end{example}

\subsection{$(C,\alpha, \beta)$-integral tetrahedral property}\label{ssec-abiTp}

The $n$-dimensional $(C,\alpha, \beta)$-integral tetrahedral property also has an integral version (cf. Definition \ref{def-integralCb}).

\begin{definition}[$(C,\alpha,\beta)$-integral tetrahedral property] \label{defn-intTetra}
Let $C>0$ and $\alpha, \beta\in (0,2)$, $\alpha < \beta$.  A metric space $(X,d)$ satisfies the $n$-dimensional $(C,\alpha,\beta)$-integral tetrahedral property at a point $p$ for radius $r$ if there exist points $p_1,\ldots, p_{n-1}\in \bar{X}$ such that $d(p,p_i)=r$ and for all $(t_1,\ldots, t_{n-1})\in [\alpha r,\beta r]^{n-1}$ the following estimate holds  
\[ 
\int_{t_1=\alpha r}^{\beta r} \cdots \int_{t_{n-1}=\alpha r}^{\beta r }
h(p,r,t_1,...t_{n-1})
\, dt_1dt_2...dt_{n-1}
\geq C (\beta- \alpha)^{n-1}r^n.
\]

We say that $X$ \textit{satisfies the $n$-dimensional $(C,\alpha,\beta)$-integral tetrahedral property for radius r} if it satisfies the  $n$-dimensional $(C,\alpha, \beta)$-integral tetrahedral property at every point for radius $r$. 
\end{definition}

Portegies-Sormani proved that the tetrahedral property implies the integral tetrahedral property. We prove a similar statement.

\begin{prop}[c.f. Proposition 3.37 in \cite{PorSor}]
If $(X,d)$ is a metric space that satisfies the $n$-dimensional $(C,\alpha,\beta)$-tetrahedral property at a point $p$ for radius $r$ 
then it also satisfies the $n$-dimensional  $(C,\alpha,\beta)$-integral tetrahedral property at $p$ for radius $r$.
\end{prop}

\begin{proof}
It follows immediately from the definitions:
\begin{eqnarray*}
\int_{t_1=\alpha r}^{\beta r} \cdots \int_{t_{n-1}=\alpha r}^{\beta r }
&h(p,r,t_1,...t_{n-1})&
\, dt_1dt_2...dt_{n-1} \,\,\,\ge\\
&\ge& \int_{t_1=\alpha r}^{\beta r} \cdots \int_{t_{n-1}=\alpha r}^{\beta r }
C r
\, dt_1dt_2...dt_{n-1}\\
&=& C( \beta - \alpha)^{n-1}r^n .
\end{eqnarray*}
\end{proof}

\begin{rmrk}\label{rmrk-modIntTimpT}
Let $(X,d)$ be a metric space that satisfies the $n$-dimensional $(C,\beta)$-integral tetrahedral property at $p$ for radius $r$. Then it satisfies the $n$-dimensional $(C, 1-\beta, 1+\beta)$-integral tetrahedral property at $p$ for radius $r$.  If $(X,d)$ satisfies the $n$-dimensional $(C, \alpha, \beta)$-tetrahedral property at $p$ for radius $r$ then it satisfies the $n$-dimensional $(C, \min\{1-\alpha, \beta - 1\})$-integral tetrahedral property at $p$ for radius $r$ if $\alpha < 1 < \beta$.  
\end{rmrk}

\begin{prop}\label{prop-EquivIntT}
Let $(X,d)$ be a metric space that satisfies the $n$-dimensional $(C, \alpha,\beta)$-integral tetrahedral property at $p$ for radius $r$. Let $\beta'=  \max\{1-\alpha, \beta - 1\}$. 
Then $(X,d)$ satisfies the $n$-dimensional $(C', \beta')$-integral tetrahedral property at $p$ for radius $r$.
\end{prop}

\begin{proof}
Note that $\beta' \in (0,1)$ and that $1-\beta' \leq \alpha < \beta \leq 1+\beta'$. Now, 
\begin{eqnarray*}
& \int_{t_1= (1-\beta')r}^{(1+ \beta' )r} \cdots \int_{t_{n-1} = (1-\beta') r}^{ (1+ \beta') r }
h(p,r,t_1,...t_{n-1})
\, dt_1dt_2...dt_{n-1}  \\
& \ge \int_{t_1=\alpha r}^{\beta r} \cdots \int_{t_{n-1}=\alpha r}^{\beta r }
h(p,r,t_1,...t_{n-1})
\, dt_1dt_2...dt_{n-1} \\
& = C( \beta - \alpha)^{n-1}r^n =  \tfrac {C( \beta - \alpha)^{n-1} }  { (2 \beta')^{n-1}}    (2 \beta')^{n-1} r^n.
\end{eqnarray*}
We conclude that $(X,d)$ satisfies the $n$-dimensional $(C', \beta')$-integral tetrahedral property at $p$ for radius $r$.
\end{proof}


\begin{example}
Proceeding as in Example \ref{ex-integralCb} it follows that the space given in 
Example \ref{ex-planes2}  satisfies the $2$-dimensional $(C, \alpha, \beta)$-integral te\-tra\-he\-dral property at $p \in X$ for any radius $r \leq \max\{\dist(p,xy-\text{plane}), \dist(p,yz-\text{plane})\}$, where $C$ is a constant that depends on $\alpha$ and $\beta$.
\end{example}


\subsection{Masses of Balls}\label{ssec-mass}

Now we  will deal with integral current spaces. We recommend the reader to check Subsection \ref{ssec-currents}
for a quick introduction to the subject and \cite{AK}, \cite{PorSor} and \cite{SW} for a complete treatment. 
In this section we prove Theorem  \ref{tetra-manifold} and the equivalent result for integral current spaces, Theorem \ref{mass-ball}. 

\begin{thm}[cf. Theorem 3.38 in \cite{PorSor}]\label{tetra-ball} 
Suppose $(X,d,T)$ is an $n$-dimensional integral current space and  $p\in X$ such that
$\bar{B}_R(p) \cap \set{\partial T}=\emptyset$.  Then for almost
every $r\in (0,R)$, if  $(\bar{X},d)$ satisfies the
$n$-dimensional 
$(C,\alpha,\beta)$-integral tetrahedral property at $p$
for radius $r$ 
then
\begin{equation}
\mass(S(p,r))\geq \SF_{n-1}(p,r) \geq C (\beta- \alpha)^{n-1}r^n.
\end{equation}
\end{thm}

\begin{proof}
Let $q_1,...,q_{n-1} \in X$, by Portegies-Sormani's Theorem \ref{dist-set-2} we know that 
\begin{align}\label{eq-SmassS}
\mass (S(p,r))\ge &\SF(p,r,q_1,...,q_{n-1})\\ \nonumber
\geq &\int_{t_1=d(p,q_1)-r}^{d(p,q_1)+r} \cdots \int_{t_{n-1}=d(p,q_{n-1})-r}^{d(p,q_{n-1})+r}
h(p,r,t_1,...t_{n-1})
\, dt_1dt_2...dt_{n-1}.
\end{align}
To get the first inequality recall that 
\begin{equation*}
\SF_{n-1}(p,r) = \sup\{\SF(p,r,q_1,...,q_{n-1}) \,|\, d(p,q_i)=r \}.
\end{equation*}
Since $S(p,r)$ does not 
depend on $q_1,...,q_{n-1}$,  we obtain from the previous series of inequalities 
\[
\mass (S(p,r)) \ge  \SF_{n-1}(p,r) .
\]

For the second inequality notice that since $(\bar{X},d)$ satisfies the
$n$-dimensional 
$(C,\alpha,\beta)$-integral tetrahedral property at $p$ for radius $r$, 
there exist  $p_1,...,p_{n-1} \in \bar X$ such that  
\[ 
\int_{t_1=\alpha r}^{\beta r} \cdots \int_{t_{n-1}=\alpha r}^{\beta r }
h(p,r,t_1,...t_{n-1})
\, dt_1dt_2...dt_{n-1}
\geq C (\beta- \alpha)^{n-1}r^n.
\]
Now $0 <\alpha< \beta <2$ implies that  $d(p,q_{i})-r < \alpha r < \beta r < d(p,q_{i})+r$. Thus, recalling equation (\ref{eq-SmassS}) we obtain
\[
\SF(p,r,p_1,...,p_{n-1})\\
\geq C (\beta- \alpha)^{n-1}r^n.
\]
\end{proof}

An immediate consequence of Theorem~\ref{tetra-ball} is Theorem \ref{tetra-manifold}.

\begin{proof}[Proof of Theorem \ref{tetra-manifold}]
It follows from the previous theorem since in this case $\mass(S(p,r))= \vol(B_r(p))$ (see Lemma 3.2 in \cite{PorSor}, cf. Section  \ref{ssec-currents}).
\end{proof}

\begin{thm}[cf. Theorem 3.42 in \cite{PorSor}]\label{mass-ball} 
Suppose $(X,d,T)$ is an $n$-dimensional integral current space and  $p\in X$ such that
$\bar{B}_R(p) \cap \set{\partial T}= \emptyset$.   If  $(\bar{X},d)$ satisfies the
$n$-dimensional $(C,\alpha,\beta)$-integral tetrahedral property at $p$
for all radii $r$,  $r \leq r_0$,  then for any
$r\in (0, \min\{r_0, R \})$
\begin{equation}
||T||(B_r(p)) \geq C (\beta- \alpha)^{n-1}r^n.
\end{equation}
\end{thm}

\begin{proof}
By Theorem \ref{tetra-ball}, we can choose $\delta_i \downarrow 0$ such that
\begin{equation}\label{eq-MSdelta}
\mass(S(p,r+\delta_i/2))  \geq \SF_{n-1}(p,r+ \delta_i/2) \geq C (\beta- \alpha)^{n-1}(r+ \delta_i/2)^n.
\end{equation}
 Thus, from 
(\ref{ball-in-ball}) and (\ref{eq-MSdelta}), 
\begin{align}
||T||(B_{r+ \delta_i}(p)) \geq  & ||T||(\bar B_{r+ \delta_i/2}(p)) \geq \mass(S(p,r+\delta_i/2)) \\
 \geq   &   C (\beta- \alpha)^{n-1}(r+ \delta_i/2)^n.
\end{align}
Taking the limit as $i \to \infty$, we get the estimate
\begin{equation}
||T||(B_{r}(p))  \geq C (\beta- \alpha)^{n-1}r^n.
\end{equation}
\end{proof}


\subsection{Convergence Theorems}\label{ssec-convergence}

Applying the mass measure estimates for balls from the previous subsection we show that the Gromov-Hausdorff and intrinsic flat convergence theorems  proven by Portegies-Sormani for the $(C, \beta)$-(integral) tetrahedral property in  \cite{PorSor} 
also hold for the $(C, \alpha, \beta)$-(integral) tetrahedral property, Theorem  \ref{tetra-compactness11}  and Theorem \ref{tetra-compactness3}. We get as corollaries  Theorem \ref{tetra-compactness2} and Theorem \ref{tetra-compactness4}.

\begin{thm}[cf. Theorem 3.42 in \cite{PorSor}]\label{tetra-compactness11}  
Let  $C>0$, $0< \alpha <  \beta < 2$,  $r_0>0$, $V_0>0$ and $(X_i, d_i, T_i)$ be a sequence of $n$-dimensional integral current spaces 
that satisfy 
\[
\mass(T_i) \le V_0,\,\,\partial T_i=0.
\]
Suppose that $(X_i,d_i)$ are compact length metric spaces that satisfy the $n$-dimensional $(C,\alpha,\beta)$-(integral) tetrahedral
property for all radii $r \leq  r_0$.  Then a subsequence of the $(X_i,d_i)$ converges in Gromov-Hausdorff sense.  In particular, there exists a constant $D_0(C,\alpha,\beta, r_0, V_0) > 0$ for which 
$ \diam(X_i) \leq D_0(C,\alpha,\beta, r_0, V_0)$.
\end{thm}

\begin{proof}
Given $\varepsilon \in (0,r_0)$, apply Theorem~\ref{tetra-ball} to get $||T||(B_\varepsilon(p)) \geq C (\beta- \alpha)^{n-1}\varepsilon^n$. 
This bound together with $\mass(T_i) \le V_0$ allows us to uniformly 
bound the maximal number of disjoint balls of radius $\varepsilon$ contained in $X_i$.  
Hence, precompactness follows from Gromov's Compactness Theorem.  Since $(X_i,d_i)$ are length metric 
spaces, the uniform bound on the maximal number of disjoint balls of radius $r_0/2$  provides a uniform upper bound
on $\diam(X_i)$ that only depends on $C, \alpha, \beta, r_0$ and $V_0$.   
\end{proof}

Theorem \ref{tetra-compactness2} follows from Theorem \ref{tetra-compactness11}.

\begin{thm}[cf. Theorem 5.2 in \cite{PorSor}]\label{tetra-compactness3}
Let $r_0>0$, $0< \alpha <  \beta < 2$,  $C, V_0>0$ and $(X_i, d_i, T_i)$ be a sequence of $n$-dimensional integral current spaces 
with 
\[
\mass(T_i) \le V_0,\,\,\partial T_i=0. 
\]
Suppose that $(X_i,d_i)$ are compact length metric spaces that satisfy the $n$-dimensional $(C,\alpha,\beta)$-(integral) tetrahedral
property for all radii $r \leq  r_0$. Then $(X_i,d_i)$ has a Gromov-Hausdorff and intrinsic flat convergent subsequence
whose limits agree. 
\end{thm}

\begin{proof}
We just need to show that the hypotheses of Theorem~\ref{SF_k-compactness} are 
satisfied.  First, since we are considering integral current spaces $X_i$ with $\bdry T_i=0$, $\mass(\bdry T_i)=0$. 
Second, by Theorem~\ref{tetra-compactness11} there
exists a uniform upper bound on the diameter of the $X_i$.  Finally,  since for all $p \in X_i$ and $r \leq r_0$,  $(X_i,d_i)$ satisfies the $n$-dimensional $(C,\alpha,\beta)$-(integral) tetrahedral
property at $p \in X_i$ for radius $r$,  Theorem \ref{tetra-ball} implies that 
\[
\SF_{n-1}(p,r) \ge C(\beta - \alpha)^{n-1} r^n.
\]
Now we can apply Theorem~\ref{SF_k-compactness} and get the result. 
\end{proof}

Theorem \ref{tetra-compactness4} follows from Theorem \ref{tetra-compactness3}.



\end{document}